\documentclass[11pt]{article}

\usepackage{amsfonts}
\usepackage{amscd}
\usepackage{amssymb}
\usepackage{amsthm}
\usepackage{amsmath} 
\usepackage{hyperref}
\usepackage{pictexwd}
\usepackage{color}

\usepackage{stmaryrd}
\usepackage{graphicx}
\usepackage{verbatim}

\theoremstyle{plain}
	\newtheorem{thm}{Theorem}[section]

\theoremstyle{definition}

	\newtheorem{remark}[thm]{Remark}
\theoremstyle{remark}

\numberwithin{equation}{section}

\def\cE{\mathcal{E}}

\def \cP{\mathcal{P}}

\def\cX{\mathcal{X}}

\def\CC{\mathbb{C}}

\def\FF{\mathbb{F}}

\def\QQ{\mathbb{Q}}
\def\RR{\mathbb{R}}

\def\ZZ{\mathbb{Z}}

\def\fa{\mathfrak{a}}

\def\fc{\mathfrak{c}}
\def\fd{\mathfrak{d}}

\def\fg{\mathfrak{g}}
\def\fh{\mathfrak{h}}

\def\fn{\mathfrak{n}}
\def\fo{\mathfrak{o}}
\def\fp{\mathfrak{p}}

\def\fr{\mathfrak{r}}

\def\fsl{\mathfrak{sl}}

\def\ad{\mathrm{ad}}
\def\Aut{\mathrm{Aut}}
\def\Card{\mathrm{Card}}

\def\dim{\mathrm{dim}}

\def\la{\lambda}

\def\mapright#1{\smash{\mathop
        {\longrightarrow}\limits^{#1}}}

\definecolor{Gray}{gray}{0.5}

\makeatletter
\renewcommand{\@makefnmark}{\mbox{\textsuperscript{}}}
\makeatother

\title{Combinatorics in affine flag varieties}
\author{
James Parkinson \\
Institut f\"{u}r Mathematische Strukturtheorie \\ Technische Universit\"{a}t Graz\\
Steyrergasse 30/III, A-8010 Graz Austria \\
parkinson@weyl.math.tu-graz.ac.at \\ \\
\and
\\
Arun Ram \\
Department of Mathematics\\ University of Wisconsin\\
Madison, WI 53706 USA \\ 
ram@math.wisc.edu \\
and \\
Department of Mathematics and Statistics \\
University of Melbourne \\
Parkville VIC 3010 Australia \\ \\
\and
\\
Cristoph Schwer \\
Mathematisches Institut Universit\"{a}t zu K\"{o}ln\\
Weyertal 86-90, 50931 K\"oln, Germany\\
cschwer@math.uni-koeln.de
}
\date{}

\begin{document}

\maketitle

$$\textit{Dedicated to Gus Lehrer on the occasion of his $60^{\mathrm{th}}$ birthday}$$

\begin{abstract}
The Littelmann path model gives a realisation of the crystals of integrable representations
of symmetrizable Kac-Moody Lie algebras.  Recent work of Gaussent-Littelmann \cite{GL}
and others \cite{BG} \cite{GR} has demonstrated a connection between this model and the 
geometry of the loop Grassmanian. The alcove walk model is a version of the path model which 
is intimately connected to the combinatorics of the affine Hecke algebra. 
In this paper we define a refined alcove walk model which encodes the points of the affine flag variety.  
We show that this combinatorial indexing naturally indexes the cells in generalized Mirkovic-Vilonen 
intersections.
\end{abstract}


\section{Introduction}

A \textit{Chevalley group} is a group in which row reduction works. This means that it is a group with a special set of generators (the ``elementary matrices'') and relations which are generalisations of the usual row reduction operations. One way to efficiently encode these generators and relations is with a Kac-Moody Lie algebra $\fg$.   From the data of the Kac-Moody Lie algebra and a choice of a 
commutative ring or field $\FF$ the group $G(\FF)$ is built by generators and relations 
following Chevalley-Steinberg-Tits.

Of particular interest is the case where $\FF$ is the field of fractions of $\fo$, the
discrete valuation ring $\fo$ is the ring of integers in $\FF$, $\fp$ is the unique maximal
ideal in $\fo$ and $k = \fo/\fp$ is the residue field.  The favourite examples are
$$\begin{array}{lll}
\FF = \CC((t)) \qquad &\fo = \CC[[t]] \qquad &k=\CC, \\
\FF = \QQ_p &\fo=\ZZ_p &k = \FF_p, \\
\FF = \FF_q((t)) & \fo = \FF_q[[t]] &k = \FF_q,
\end{array} 
$$
where $\QQ_p$ is the field of $p$-adic numbers, $\ZZ_p$ is the ring of $p$-adic integers,
and $\FF_q$ is the finite field with $q$ elements.  For clarity of presentation we shall work in the
first case where $\FF=\CC((t))$.  The diagram
\begin{align}\label{eq:picture}
\begin{matrix}\FF\\
$\beginpicture
\setcoordinatesystem units <0.8cm,0.8cm>         % sets scale
\setplotarea x from -0.1 to 0.3, y from -0.5 to 0.5  % sets plot size up
\put{$\cup$} at 0 0 \put{$\shortmid$} at 0.21 0.08
\put{$\shortmid$} at 0.21 0.0
\endpicture$  \\
\fo&\mapright{\mathrm{ev}_{t=0}} & k=\fo/\fp
\end{matrix}\qquad\quad\textrm{gives}\quad\qquad
\begin{matrix}
G &= &G(\CC((t))) \\
$\beginpicture
\setcoordinatesystem units <0.8cm,0.8cm>         % sets scale
\setplotarea x from -0.1 to 0.3, y from -0.5 to 0.5  % sets plot size up
\put{$\cup$} at 0 0 \put{$\shortmid$} at 0.21 0.08
\put{$\shortmid$} at 0.21 0.0
\endpicture$ &&$\beginpicture
\setcoordinatesystem units <0.8cm,0.8cm>         % sets scale
\setplotarea x from -0.1 to 0.3, y from -0.5 to 0.5  % sets plot size up
\put{$\cup$} at 0 0 \put{$\shortmid$} at 0.21 0.08
\put{$\shortmid$} at 0.21 0.0
\endpicture$ \\
K &=& G(\CC[[t]]) &\mapright{\mathrm{ev}_{t=0}} &G(\CC) \\
$\beginpicture
\setcoordinatesystem units <0.8cm,0.8cm>         % sets scale
\setplotarea x from -0.1 to 0.3, y from -0.5 to 0.5  % sets plot size up
\put{$\cup$} at 0 0 \put{$\shortmid$} at 0.21 0.08
\put{$\shortmid$} at 0.21 0.0
\endpicture$ &&$\beginpicture
\setcoordinatesystem units <0.8cm,0.8cm>         % sets scale
\setplotarea x from -0.1 to 0.3, y from -0.5 to 0.5  % sets plot size up
\put{$\cup$} at 0 0 \put{$\shortmid$} at 0.21 0.08
\put{$\shortmid$} at 0.21 0.0
\endpicture$ &&$\beginpicture
\setcoordinatesystem units <0.8cm,0.8cm>         % sets scale
\setplotarea x from -0.1 to 0.3, y from -0.5 to 0.5  % sets plot size up
\put{$\cup$} at 0 0 \put{$\shortmid$} at 0.21 0.08
\put{$\shortmid$} at 0.21 0.0
\endpicture$ \\
I &= &\mathrm{ev}_{t=0}^{-1}(B(\CC)) &\mapright{\mathrm{ev}_{t=0}} &B(\CC)
\end{matrix}
\end{align}
where $B(\CC)$ is the ``Borel subgroup'' of ``upper triangular matrices'' in $G(\CC)$.
The \textit{loop group} is $G=G(\CC((t)))$, $I$ is the standard \textit{Iwahori subgroup} of $G$, 
\begin{equation}
\begin{array}{c}
G(\CC)/B(\CC)\quad\textrm{is the \textit{flag variety}}, \\ \\
G/I\quad\textrm{is the \textit{affine flag variety}},\qquad\textrm{and}\qquad G/K\quad\textrm{is the \textit{loop Grassmanian}}.
\end{array}
\end{equation}
The primary tool for the study of these varieties (ind-schemes) are the following
``classical'' double coset decompositions, see \cite[Ch.\,8]{St} and \cite[\S (2.6)]{Mac1}

\begin{thm}\label{thm:decompositions} Let $W$ be the Weyl group of $G(\CC)$,
$\widetilde{W} = W \ltimes \fh_\ZZ$ the affine Weyl group, and $U^-$ the subgroup
of ``unipotent lower triangular'' matrices in $G(\FF)$ and $\fh_\ZZ^+$ the set of
dominant elements of $\fh_\ZZ$.  Then
$$\begin{array}{llcll}
\begin{array}{c} \hbox{Bruhat} \\ \hbox{decomposition}\end{array}
\quad&\displaystyle{G = \bigsqcup_{w\in W} BwB}
&&\displaystyle{K=\bigsqcup_{w\in W}IwI}
\\
\\
\\
\begin{array}{c} \textrm{Iwahori} \\ \textrm{decomposition}\end{array}
\quad&\displaystyle{G = \bigsqcup_{w\in \widetilde{W}} IwI}
&&\displaystyle{G = \bigsqcup_{v\in \widetilde{W}} U^-vI}
\\
\\
\\
\begin{array}{c} \hbox{Cartan} \\ \hbox{decomposition}\end{array}\quad
&\displaystyle{G= \bigsqcup_{\lambda^\vee\in \fh_\ZZ^+}
Kt_{\lambda^\vee} K } &\quad &\displaystyle{G =
\bigsqcup_{\mu^\vee\in \fh_\ZZ} U^-t_{\mu^\vee} K}
&\begin{array}{c}\hbox{Iwasawa} \\ \hbox{decomposition}
\end{array}
\end{array}
$$
\end{thm}

In this paper we shall refine the Littelmann path model (in its alcove walk form, see
[Ra]) by putting labels on
the paths to provide a combinatorial indexing of the points in the affine flag variety.
This combinatorial method of expressing the points of $G/I$
gives detailed information about the structure of the intersections
\begin{align}\label{eq:intersections}
U^-vI\cap IwI\qquad\textrm{with}\qquad v,w\in\widetilde{W}.
\end{align}
The corresponding intersections in $G/K$ have arisen in many contexts.
Most notably, 
the set of \textit{Mirkovi\'{c}-Vilonen cycles of shape $\la^{\vee}$ and weight $\mu^{\vee}$} is
the set of irreducible components of the closure of $U^-t_{\mu^{\vee}}K\cap Kt_{\la^{\vee}}K$
in $G/K$,
$$
MV(\la^{\vee})_{\mu^{\vee}}=\mathrm{Irr}(\overline{U^-t_{\mu^{\vee}}K\cap Kt_{\la^{\vee}}K}),
$$
and
$$
\hbox{when $k=\FF_q$,}\qquad
\Card_{G/K}(U^-t_{\mu^{\vee}}K\cap Kt_{\la^{\vee}}K)\ \ \hbox{is}
$$
(up to some easily understood factors) 
the coefficient of the monomial symmetric function $m_{\mu^{\vee}}$ in the expansion 
of the Macdonald spherical function $P_{\la^{\vee}}$.

The research of A.\ Ram and J.\ Parkinson was partially supported by the National
Science Foundation under grant DMS-0353038 at the University of Wisconsin.  
The research of C.\ Schwer was supported by a fellowship within the Postdoc-Programme of the German Academic Exchange Service (DAAD). J.\ Parkinson and and C.\ Schwer
thank the University of Wisconsin, Madison for hospitality.  This paper was stimulated
by the workshop on \emph{Buildings and Combinatorial Representation Theory} at
the American Institute of Mathematics March 26-30, 2007.  
We thank these institutions for support of our research.

\section{Borcherds-Kac-Moody Lie algebras}

This section reviews definitions and sets notations for Borcherds-Kac-Moody Lie algebras.
Standard references are the book of Kac \cite{Kac}, 
the books of Wakimoto \cite{Wak1}\cite{Wak2}, the survey article of Macdonald
\cite{Mac4} and the handwritten notes of Macdonald \cite{Mac3}.  Specifically,
\cite[Ch.\,1]{Kac} is a reference for \S 2.1, \cite[Ch.\,3 and 5]{Kac} for \S 2.2, and
\cite[Ch.\,2]{Kac} for \S 2.3.

\subsection{Constructing a Lie algebra from a matrix}

Let $A = (a_{ij})$ be an $n\times n$ matrix.  Let
\begin{equation}
r = \mathrm{rank}(A),
\qquad
\ell = \mathrm{corank}(A),
\qquad\hbox{so that}\quad
r+\ell = n.
\end{equation}
By rearranging rows and columns we may assume that $(a_{ij})_{1\le i,j\le r}$ is nonsingular.  
Define a $\CC$-vector space
\begin{equation}\label{fhdef}
\fh = \fh' \oplus \fd, \qquad\hbox{where}
\qquad
\begin{array}{l}
\hbox{$\fh'$ has basis $h_1,\ldots, h_n$,\quad and} \\
\hbox{$\fd$ has basis $d_1,\ldots, d_\ell$.}
\end{array}
\end{equation}
Define $\alpha_1,\ldots, \alpha_n\in \fh^*$ by
\begin{equation}
\alpha_i(h_j) = a_{ij}
\qquad\hbox{and}\qquad
\alpha_i(d_j) = \delta_{i,r+j},
\end{equation}
and let
\begin{equation}\label{fcdef}
\bar\fh' = \fh'/\fc,
\qquad\hbox{where}\qquad
\fc = \{ h\in \fh'\ |\ \hbox{$\alpha_i(h)=0$ for all $1\le i\le n$}\}.
\end{equation}
Let $c_1\ldots, c_\ell\in \fh'$ be a basis of $\fc$ so that $h_1,\ldots, h_r, c_1,\ldots, c_\ell,d_1,\ldots,d_{\ell}$ is another basis of $\fh$ and define $\kappa_1,\ldots, \kappa_\ell\in \fh^*$
by
\begin{equation}
\kappa_i(h_j)=0,
\qquad
\kappa_i(c_j) = \delta_{ij},
\qquad\hbox{and}\qquad
\kappa_i(d_j) = 0.
\end{equation}
Then $\alpha_1,\ldots, \alpha_n, \kappa_1,\ldots, \kappa_\ell$ form a basis of $\fh^*$.
%Pictorially, these definitions are represented by the matrix
%$$\begin{matrix}
%&h_1,\ldots, h_r \ \ c_1,\ldots, c_\ell\ \  d_1,\ldots, d_\ell \\
%\begin{matrix} 
%\alpha_1 \\ \vdots \\ \alpha_r \\
%\alpha_{r+1} \\ \vdots \\ \alpha_n \\
%\omega_1 \\ \vdots \\ \omega_\ell 
%\end{matrix} 
%&\begin{pmatrix}
%\beginpicture
%\setcoordinatesystem units <0.9cm,0.9cm>         % sets scale
%\setplotarea x from -3 to 3, y from -3 to 3  % sets plot size up
%    \plot -1 3 -1 -3  /
%    \plot 1 3 1 -3  /
%    \plot 3 -1 -3  -1 /
%    \plot 3 1 -3 1 /
%\put{$a_{ij}$} at -2 2 
%\put{$0$} at 0 2
%\put{$0$} at 2 2
%\put{$a_{ij}$} at -2 0
%\put{$0$} at 0 0 
%\put{$\begin{matrix} 1 \\ &\ddots \\ &&1\end{matrix}$} at 2 0
%\put{$0$} at -2 -2
%\put{$\begin{matrix} 1 \\ &\ddots \\ &&1\end{matrix}$} at 0 -2
%\put{$0$} at 2 -2
%\endpicture
%\end{pmatrix}
%\end{matrix}
%$$

Let $\fa$ be the Lie algebra given by generators $\fh, e_1,\ldots, e_n, f_1,\ldots, f_n$ and 
relations
\begin{equation}
[h,h'] = 0, \qquad
[e_i, f_j] = \delta_{ij}h_i, \qquad
[h,e_i] = \alpha_i(h) e_i, \qquad
[h, f_i] = -\alpha_i(h) f_i,
\end{equation}
for $h,h'\in \fh$ and $1\leq i,j\leq n$.  The \emph{Borcherds-Kac-Moody Lie algebra} of $A$ is 
\begin{equation}\label{Serreideal}
\fg = \frac{\fa}{\fr},
\qquad\hbox{where}\quad
\hbox{ $\fr$ is a the largest ideal of $\fa$ such that $\fr\cap \fh=0$.}
\end{equation}
The Lie algebra $\fa$ is graded by 
\begin{equation}
Q = \sum_{i=1}^{n} \ZZ \alpha_i,
\qquad\hbox{by setting}\quad
\deg(e_i)=\alpha_i,\ \ \deg(f_i) = -\alpha_i,\ \ \deg(h)=0,
\end{equation}
for $h\in \fh$.  Any ideal of $\fa$ is $Q$-graded and so $\fg$ is $Q$-graded
(see \cite[(1.6)]{Mac3} or \cite[p.\ 81]{Mac4}),
\begin{align}
\label{rootdef}
\begin{aligned}
\fg = \fg_0 \oplus \left(\bigoplus_{\alpha\in R} \fg_\alpha\right),
\qquad\hbox{where}\quad
\fg_{\alpha}=\{x\in\fg\mid[h,x]=\alpha(h)x\},\quad\textrm{and}\\
R = \{ \alpha\ |\ \hbox{$\alpha\ne 0$ and $\fg_\alpha\ne 0$}\}
\qquad\textrm{is the set of \emph{roots} of $\fg$.}
\end{aligned}
\end{align}
The \emph{multiplicity} of a root $\alpha\in R$ is 
$\dim(\fg_\alpha)$ and the decomposition of $\fg$ in \eqref{rootdef} is the decomposition
of $\fg$ as an $\fh$-module (under the adjoint action).  If
$$\begin{array}{l}
\hbox{$\fn^+$ is the subalgebra generated by $e_1,\ldots, e_n$, and} \\
\hbox{$\fn^-$ is the subalgebra generated by $f_1,\ldots, f_n$,} 
\end{array}
$$
then (see \cite[p.\ 83]{Mac4} or \cite[\S 1.3]{Kac})
\begin{equation}\label{KMtridecomp}
\fg = \fn^-\oplus\fh \oplus \fn^+
\qquad\hbox{and}\qquad
\fh = \fg_0, \quad
\fn^+ = \bigoplus _{\alpha\in R^+} \fg_\alpha,
\qquad
\fn^- = \bigoplus_{\alpha\in R^+} \fg_{-\alpha},
\end{equation}
where
\begin{equation}
R^+ = Q^+\cap R
\qquad\hbox{with}\qquad
Q^+ = \sum_{i=1}^n \ZZ_{\ge 0}\alpha_i.
\end{equation}

Let $\fc$ and $\fd$ be as in \eqref{fhdef} and \eqref{fcdef}.  Then 
$$\hbox{$\fd$ acts on $\fg'=[\fg,\fg]$ by derivations,}
\qquad
\fc = Z(\fg) = Z(\fg'),$$
\begin{equation}\label{KMLiealgs}
\begin{array}{l}
\fg= \fn^- \oplus \fh \oplus \fn^+ = \fa/\fr = \fg' \rtimes \fd,
\\ \\
\fg' 
= \fn^- \oplus \fh' \oplus \fn^+
= [\fg, \fg], 
\\ \\
\bar\fg' 
= \fn^-\oplus \bar\fh' \oplus \fn^+
= \fg'/\fc, 
\end{array}
\end{equation}
and $\fg'$ is the universal central extension of $\fg'$ (see 
\cite[Ex. 3.14]{Kac}). 

\subsection{Cartan matrices, $\fsl_2$ subalgebras and the Weyl group}

A \emph{Cartan matrix} is an $n\times n$ matrix $A = (a_{ij})$ such that
\begin{equation}
a_{ij}\in \ZZ,\qquad
a_{ii} = 2, \qquad
\hbox{$a_{ij} \le 0$ if $i\ne j$},
\qquad
\hbox{$a_{ij}\ne 0$ if and only if $a_{ji}\ne 0$}.
\end{equation}
When $A$ is a Cartan matrix the Lie algebra $\fg$ contains many subalgebras
isomorphic to $\fsl_2$.  For $1\le i\le n$, the elements $e_i$ and $f_i$ act locally
nilpotently on $\fg$ (see \cite[p.\ 85]{Mac4} or \cite[(1.19)]{Mac3} or \cite[Lemma~3.5]{Kac}),
\begin{equation}
\mathrm{span}\{e_i,f_i,h_i\} \cong \fsl_2,
\qquad\hbox{and}\qquad
\tilde s_i = \exp(\ad e_i)\exp(-\ad f_i)\exp(\ad e_i)
\end{equation}
is an automorphism of $\fg$ (see \cite[Lemma~3.8]{Kac}).  Thus $\fg$ has lots of symmetry.

The \emph{simple reflections} $s_i\colon \fh^*\to \fh^*$ and $s_i\colon \fh \to \fh$ are given by
\begin{equation}\label{si}
s_i\lambda = \lambda - \lambda(h_i)\alpha_i
\qquad\hbox{and}\qquad
s_ih = h- \alpha_i(h)h_i,
\qquad\hbox{for $1\le i\le n$,}
\end{equation}
$\lambda\in \fh^*$, $h\in \fh$, and 
$$\tilde s_i \fg_\alpha = \fg_{s_i\alpha}
\qquad\hbox{and}\qquad
\tilde s_i h = s_i h,
\qquad\hbox{for $\alpha\in R,\ \  h\in \fh$}.
$$
The \emph{Weyl group} $W$ is the subgroup of $GL(\fh^*)$ (or $GL(\fh)$) generated by the simple reflections.
The simple reflections on $\fh$ are reflections in the hyperplanes
\begin{equation*}
\fh^{\alpha_i} = \{ h\in \fh\ |\ \alpha_i(h) = 0\},
\qquad\hbox{and}\qquad
\fc =\fh^W = \bigcap_{i=1}^n \fh^{\alpha_i}.
\end{equation*}
The representation of $W$ on $\fh$ and $\fh^*$ are dual so that
$$\lambda(wh) = (w^{-1}\lambda)(h),
\qquad \hbox{for $w\in W$, $\lambda\in \fh^*$, $h\in \fh$.}
$$
The group $W$ is presented by generators $s_1,\ldots, s_n$ and relations
\begin{equation}\label{Wpres}
s_i^2 = 1
\qquad\hbox{and}\qquad
\underbrace{s_is_js_i\cdots}_{m_{ij}\ \mathrm{factors}}
=\underbrace{s_js_is_j\cdots}_{m_{ij}\ \mathrm{factors}}
\end{equation}
for pairs $i\ne j$ such that $a_{ij}a_{ji}<4$, where $m_{ij}=2,3,4,6$ if
$a_{ij}a_{ji}=0,1,2,3$, respectively (see \cite[(2.12)]{Mac3} or \cite[Prop.\,3.13]{Kac}).

The \emph{real roots} of $\fg$ are the elements of the set
\begin{equation}
R_{\mathrm{re}} = \bigcup_{i=1}^n W \alpha_i,
\qquad\hbox{and}\qquad
R_{\mathrm{im}} = R \backslash R_{\mathrm{re}}
\end{equation}
is the set of \emph{imaginary roots} of $\fg$.  If $\alpha = w\alpha_i$ is a real root then
there is a subalgebra isomorphic to $\fsl_2$ spanned by
\begin{equation}\label{alphasl2}
e_\alpha = \tilde w e_i,\ \ 
f_\alpha = \tilde w f_i,\ \
\textrm{and}\ \ 
h_\alpha = \tilde w h_i,
\end{equation}
and $s_\alpha = ws_iw^{-1}$ is a reflection in $W$ acting on $\fh$ and $\fh^*$ by
\begin{equation}\label{reflections}
s_\alpha\lambda = \lambda- \lambda(h_\alpha)\alpha
\qquad\hbox{and}\qquad
s_\alpha h = h - \alpha(h) h_\alpha,\qquad\hbox{respectively.}
\end{equation}

Let $\fh_{\RR}=\RR\hbox{-span}\{h_1,\ldots, h_n, d_1,\ldots, d_\ell\}$.  The group $W$ 
acts on $\fh_{\RR}$ and the \textit{dominant chamber} 
\begin{equation}\label{domchamber}
C=\{\la^{\vee}\in\fh_{\RR}\mid\langle \alpha_i,\la^{\vee}\rangle\geq0\textrm{ for all $1\leq i\leq n$}\}
\end{equation}
is a fundamental domain for the action of $W$ on the \emph{Tits cone}
\begin{equation}\label{Titscone}
X=\bigcup_{w\in W}wC=\{h\in\fh_{\RR}\mid\langle\alpha,h\rangle<0
\textrm{ for a finite number of $\alpha\in R^+$}\}.
\end{equation}
$X=\fh_{\RR}$ if and only if $W$ is finite (see \cite[Prop.\,3.12]{Kac} and \cite[(2.14)]{Mac3}).

\subsection{Symmetrizable matrices and invariant forms}

A \emph{symmetrizable matrix} is a matrix $A = (a_{ij})$ such that there exists a diagonal
matrix 
\begin{equation}\label{symmetrizable}
\cE = \mathrm{diag}(\epsilon_1,\ldots, \epsilon_n),\ \ \epsilon_i\in \RR_{>0},
\qquad\hbox{such that}\qquad
\hbox{$A \cE$ is symmetric}.
\end{equation}

If $\langle,\rangle\colon \fg\times \fg\to \CC$ is a $\fg$-invariant symmetric bilinear form then
$$\langle h_i ,h \rangle
=\langle [e_i,f_i], h\rangle
= - \langle f_i, [e_i,h]\rangle
= \langle f_i, \alpha_i(h)e_i\rangle
= \alpha_i(h)\langle e_i, f_i\rangle,
$$
so that
\begin{equation}\label{invform}
\langle h_i, h\rangle = \alpha_i(h) \epsilon_i,
\qquad\hbox{where}\qquad
\epsilon_i = \langle e_i, f_i\rangle.
\end{equation}
Conversely, if $A$ is a symmetrizable matrix then there is a nondegenerate invariant
symmetric bilinear form on $\fg$ determined by the formulas in \eqref{invform} (see
\cite[(3.12)]{Mac3} or \cite[Theorem~2.2]{Kac}).

If $A$ is a Cartan matrix and 
$\langle,\rangle\colon \fh\times \fh\to \CC$ is a $W$-invariant symmetric bilinear form then
$$\langle h_i,h\rangle
= - \langle s_ih_i, h\rangle
= - \langle h_i, s_ih\rangle
= -\langle h_i, h-\alpha_i(h)h_i\rangle
= - \langle h_i, h\rangle + \alpha_i(h)\langle h_i,h_i\rangle,
$$
so that
\begin{equation}\label{Winvtform}
\langle h_i,h\rangle = \alpha_i(h)\epsilon_i, 
\qquad\hbox{where}\qquad
\epsilon_i = \hbox{$\frac12$}\langle h_i,h_i\rangle.
\end{equation}
In particular, $\alpha_i(h_j)\epsilon_i = \langle h_i, h_j\rangle
=\langle h_j, h_i\rangle = \alpha_j(h_i)\epsilon_j$ so that $A$ is symmetrizable.
Conversely, if $A$ is a symmetrizable Cartan matrix then there is a nondegenerate
$W$-invariant symmetric bilinear form on $\fh$ determined by the formulas in \eqref{Winvtform}
(see \cite[(2.26)]{Mac3}).

If $x_\alpha\in \fg_\alpha$, $y_\alpha\in \fg_{-\alpha}$ then
$[x_\alpha, y_\alpha]\in [\fg_\alpha, \fg_{-\alpha}]\subseteq \fg_0 = \fh$
and 
$\langle h, [x_\alpha, y_\alpha]\rangle
= - \langle [x_\alpha,h], y_\alpha\rangle
= \alpha(h) \langle x_\alpha, y_\alpha\rangle,
$
so that
\begin{equation}\label{bracketform}
[x_\alpha, y_\alpha] = \langle x_\alpha, y_\alpha\rangle h_\alpha^\vee,
\qquad\hbox{where $\langle h, h_\alpha^\vee\rangle = \alpha(h)$ for all $h\in \fh$}
\end{equation}
determines $h^\vee_\alpha\in \fh$. 
If $\alpha\in R_{\mathrm{re}}$ and $e_\alpha, f_\alpha, h_\alpha$ are as in \eqref{alphasl2} then 
\begin{equation}\label{halphavee}
h_\alpha = [e_\alpha,f_\alpha] = \langle e_\alpha,f_\alpha\rangle h_\alpha^\vee
\qquad\hbox{and}\qquad
\langle e_\alpha, f_\alpha\rangle = \hbox{$\frac12$}\langle h_\alpha, h_\alpha\rangle.
\end{equation}
Let
\begin{equation}\label{alphavee}
\alpha^\vee = \langle e_\alpha, f_\alpha\rangle\alpha 
= \hbox{$\frac12$}\langle h_\alpha, h_\alpha\rangle\alpha
\qquad\hbox{so that}\qquad
\alpha^\vee(h) 
%= \langle e_\alpha, f_\alpha\rangle \alpha(h) 
%= \langle h, \langle e_\alpha, f_\alpha\rangle h_\alpha^\vee\rangle 
= \langle h, h_\alpha\rangle.
\end{equation}
Use the vector space isomorphism 
\begin{equation}
\begin{matrix}
\fh &\mapright{\sim} &\fh^* \\
h &\longmapsto &\langle h, \cdot\rangle \\
h_\alpha &\longmapsto &\alpha^\vee \\
h_\alpha^\vee &\longmapsto &\alpha
\end{matrix}
\qquad
\hbox{to identify}\quad
Q^\vee = \sum_{i=1}^n \ZZ h_i
\quad\hbox{and}\quad
Q^* = \sum_{i=1}^n \ZZ\alpha_i^\vee
\end{equation}
and write
\begin{equation}
\langle \lambda^\vee, \mu\rangle = \mu(h_\lambda)
\qquad\hbox{if}\quad
\lambda^\vee = \lambda_1\alpha_1^\vee+\cdots+\lambda_n\alpha_n^\vee
\quad\hbox{and}\quad
h_\lambda =\lambda_1h_1+\cdots +\lambda_nh_n.
\end{equation}

%Let $\fr$ be as in \eqref{Serreideal}.  If $A$ is a Cartan matrix the elements
%\begin{equation}\label{Serrerelations}
%(\ad e_i)^{1-a_{ij}}(e_j)
%\quad\hbox{and}\quad
%(\ad f_i)^{1-a_{ij}}(f_j)
%\qquad\hbox{are in $\fr$}
%\end{equation}
%and if $A$ is a symmetrizable Cartan matrix these elements
%generate the ideal $\fr$ (see \cite[Theorem~9.11]{Kac}).

\section{Steinberg-Chevalley groups}

This section gives a brief treatment of the theory of Chevalley groups.  The
primary reference is \cite{St} and the extensions to the Kac-Moody case are 
found in \cite{Ti1}.

Let $A$ be a Cartan matrix and let $R_{\mathrm{re}}$ be the real roots of the corresponding
Borcherds-Kac-Moody Lie algebra $\fg$.  Let $U$ be the enveloping algebra of $\fg$.
For each $\alpha\in R_{\mathrm{re}}$ fix a choice of 
$e_\alpha$ in \eqref{alphasl2} (a choice of $\tilde w$).  Use the notation
$$x_\alpha(t) = \exp(te_\alpha) = 1+e_\alpha+\frac{1}{2!}t^2e_\alpha^2
+\frac{1}{3!}t^3e_\alpha^3+\cdots,
\qquad\hbox{in $U[[t]]$.}
$$
Then
$$x_\alpha(t)x_\alpha(u) = x_{\alpha}(t+u)\qquad\hbox{in $U[[t,u]]$.}$$
Following \cite[3.2]{Ti1}, 
a \textit{prenilpotent pair} is a pair of roots $\alpha,\beta\in R_{\mathrm{re}}$ 
such that there exists $w,w'\in W$ with 
$$w\alpha,w\beta\in R_{\mathrm{re}}^+
\qquad\hbox{and}\qquad
w'\alpha,w'\beta\in-R_{\mathrm{re}}^+.$$
This condition guarantees that the Lie subalgebra of $\fg$ generated by 
$\fg_{\alpha}$ and $\fg_{\beta}$ is nilpotent.  Let $\alpha,\beta$ be a prenilpotent pair
and let $e_\alpha\in \fg_\alpha$ and 
$e_\beta\in \fg_\beta$ be as in \eqref{alphasl2}.  By \cite[Lemma 15]{St}
there are unique integers 
$C^{i,j}_{\alpha\beta}$ such that 
\begin{equation*}
x_\alpha(t)x_\beta(u)
= x_\beta(u)x_\alpha(t)x_{\alpha+\beta}(C_{\alpha,\beta}^{1,1}tu) 
x_{2\alpha+\beta}(C^{2,1}_{\alpha,\beta}t^2u)
x_{\alpha+2\beta}(C^{1,2}_{\alpha,\beta}ut^2)\cdots .
\end{equation*}
%(see also \cite[p.\,5]{CC}?????).

Let $\FF$ be a commutative ring.
The \textit{Steinberg group}  
$$\hbox{$\mathrm{St}$ 
is given by generators $x_{\alpha}(f)$ for $\alpha\in R_{\mathrm{re}}$, $f\in \FF$,}$$
and relations
\begin{equation}
x_{\alpha}(f_1)x_{\alpha}(f_2)=x_{\alpha}(f_1+f_2),
\qquad\hbox{for $\alpha\in R_{\mathrm{re}}$,\quad and}
\end{equation}
\begin{equation}\label{commutatorreln}
x_\alpha(f_1)x_\beta(f_2)
= x_\beta(f_2)x_\alpha(f_1)x_{\alpha+\beta}(C_{\alpha,\beta}^{1,1}f_1f_2) x_{2\alpha+\beta}(C^{2,1}_{\alpha,\beta}f_1^2f_2)
x_{\alpha+2\beta}(C^{1,2}_{\alpha,\beta}f_1f_2^2)\cdots 
\end{equation}
for prenilpotent pairs $\alpha,\beta$.
In $\mathrm{St}$ define
\begin{equation}\label{nalphadefn}
n_{\alpha}(g)=x_{\alpha}(g)x_{-\alpha}(-g^{-1})x_{\alpha}(g),
\quad
n_\alpha = n_\alpha(1),
\quad\textrm{and}\quad 
h_{\alpha^{\vee}}(g)=n_{\alpha}(g)n_{\alpha}^{-1},
\end{equation}
for $\alpha\in R_{\mathrm{re}}$ and $g\in\FF^{\times}$.

Let $\fh_\ZZ$ be a $\ZZ$-lattice in $\fh$ which is stable under the $W$-action and such that
$$\fh_\ZZ \supseteq Q^\vee,
\qquad\hbox{where}\quad
Q^\vee = \ZZ\hbox{-span}\{h_1,\ldots, h_n\}$$
with $h_1,\ldots, h_n$ as in \eqref{fhdef}.  
With
$$\hbox{$T$ given by generators $h_{\la^{\vee}}(g)$ 
for $\la^{\vee}\in \fh_\ZZ$, $g\in\FF^{\times}$,\quad
and relations}$$
\begin{equation}\label{Tdefn}
h_{\la^{\vee}}(g_1)h_{\la^{\vee}}(g_2)=h_{\la^{\vee}}(g_1g_2)
\qquad\hbox{and}\qquad
h_{\la^{\vee}}(g)h_{\mu^{\vee}}(g)=h_{\la^{\vee}+\mu^{\vee}}(g),
\end{equation}
the \textit{Tits group}
$$\hbox{$G$ is the 
group generated by $St$ and $T$}$$
with the relations coming from the third
equation in \eqref{nalphadefn} and the additional relations
\begin{equation}\label{Titsgroup}
h_{\la^{\vee}}(g)x_{\alpha}(f)h_{\la^{\vee}}(g)^{-1}
=x_{\alpha}(g^{\langle\la^{\vee},\alpha\rangle}f) 
\qquad\hbox{and}\qquad
n_ih_{\la^{\vee}}(g)n_i^{-1}=h_{s_i\la^{\vee}}(g). 
\end{equation}

For $\alpha,\beta\in R_{\mathrm{re}}$ let $\epsilon_{\alpha\beta}=\pm1$  be given by
$$
\tilde s_{\alpha}(e_{\beta})=\epsilon_{\alpha\beta} e_{s_{\alpha}\beta},
\qquad\textrm{where}\qquad 
\tilde s_{\alpha} = \exp(\ad e_\alpha)\exp(-\ad f_\alpha)\exp(\ad e_\alpha)
$$
(see \cite[p.48]{CC} and \cite[(3.3)]{Ti1}).  
By \cite[Lemma~37]{St} (see also \cite[\S 3.7(a)]{Ti1}) 
\begin{equation}
\label{eq:R4}
n_{\alpha}(g)x_{\beta}(f)n_{\alpha}(g)^{-1}
=x_{s_{\alpha}\beta}(\epsilon_{\alpha\beta} g^{-\langle\beta,\alpha^{\vee}\rangle}f),  
\qquad
h_{\lambda^{\vee}}(g)x_{\beta}(f)h_{\lambda^{\vee}}(g)^{-1}
=x_{\beta}(g^{\langle\beta,\lambda^{\vee}\rangle}f),
\end{equation}
\begin{equation}
\label{eq:R5}
\hbox{and}\quad
n_{\alpha}(g)h_{\lambda^{\vee}}(g')n_{\alpha}(g)^{-1}
=h_{s_{\alpha}\lambda^\vee}(g').
\end{equation}
Thus
$G$ has a symmetry under the subgroup
\begin{equation}\label{Ndefn}
\hbox{$N$ generated by $T$ and the $n_{\alpha}(g)$ 
for $\alpha\in R_{\mathrm{re}}, g\in\FF^{\times}.$}
\end{equation}
If $\FF$ is big enough then $N$ is the normalizer of $T$ in $G$ \cite[Ex.\,(b) p.\,36]{St} and, by 
\cite[Lemma 27]{St},  the homomorphism
\begin{equation}\label{NtoW}
\begin{matrix}
N &\longrightarrow &W\\
n_\alpha(g) &\longmapsto &s_\alpha
\end{matrix}
\qquad\hbox{is surjective with kernel $T$.}
\end{equation}

\begin{remark} \cite[\S 3.7(b)]{Ti1}
If $\fh_\ZZ = Q^\vee$ and the first relation of \eqref{Titsgroup} holds in $\mathrm{St}$ then
there is a surjective homomorphism $\psi\colon \mathrm{St}\twoheadrightarrow G$.
By \cite[Lemma 22]{St}, the elements 
$$
n_{\alpha}h_{\la^{\vee}}(g)n_{\alpha}^{-1}h_{s_{\alpha}\la^{\vee}}(g)^{-1}
\qquad\textrm{and}\qquad n_{\alpha}(g)n_{\alpha}^{-1}h_{\alpha^{\vee}}(g)^{-1}
$$
automatically commute with each $x_{\beta}(f)$ so that $\ker(\psi)\subseteq Z(\mathrm{St})$. 
In many cases $\mathrm{St}$ is the universal central extension of $G$ (see \cite[3.7(c)]{Ti1} and \cite[Theorems~10,11,12]{St}).
\end{remark}

\begin{remark}\label{Chevgp}
The algebra $\fg'=[\fg,\fg]$ in \eqref{KMLiealgs} is generated by
$e_\alpha$, $\alpha\in R_{\mathrm{re}}$.
A $\fg'$-module $V$ is \textit{integrable} 
if $e_{\alpha}$, $\alpha\in R_{\mathrm{re}}$, act locally nilpotently so that
\begin{equation}\label{xalphaops}
x_{\alpha}(c)=\exp(ce_{\alpha}),
\qquad\textrm{for $\alpha\in R_{\mathrm{re}}$, $c\in \CC$},
\end{equation}
are well defined operators on $V$.  The
\emph{Chevalley group} 
$G_V$ is the subgroup of $GL(V)$ generated by 
the operators in \eqref{xalphaops}.
To do this integrally use a Kostant $\ZZ$-form and choose a 
lattice in the module $V$ (see \cite[\S 4.3-4]{Ti1} and \cite[Ch.\,1]{St}).
The \textit{Kac-Moody group} is the group $G_{KM}$ generated by symbols
$$
x_{\alpha}(c),\quad \alpha\in R_{\mathrm{re}}, c\in\CC,
\qquad\textrm{with relations}\qquad 
x_{\alpha}(c_1)x_{\alpha}(c_2)=x_{\alpha}(c_1+c_2)
$$
and the additional relations coming from forcing an element to be $1$ if it acts by $1$ on \emph{every} integrable $\fg'$ module.  This is essentially the Chevalley group $G_V$ for the case when $V$ is the
adjoint representation and so $G_{KM}\subseteq\Aut(\fg')$.  
There are surjective homomorphisms
$$\mathrm{St}(\CC)\twoheadrightarrow G_{KM} \twoheadrightarrow G_V.$$
See \cite[Exercises 3.16-19]{Kac} and \cite[Proposition~1]{Ti1}.
\end{remark}

\begin{remark} \cite[Lemma 28]{St} In the setting of Remark \ref{Chevgp}
let $T_V$ be the subgroup of $G_V$ generated by
$h_{\alpha^{\vee}}(g)$ for $\alpha\in R_{\mathrm{re}}, g\in\FF^{\times}$.
Then  
\begin{align*}
h_{\alpha_1^{\vee}}(g_1)\cdots h_{\alpha_n^{\vee}}(g_n)&=1
\quad\textrm{if and only if}\quad 
g_1^{\langle\mu,\alpha_1^{\vee}\rangle}\cdots g_n^{\langle\mu,\alpha_n^{\vee}\rangle}=1\quad\textrm{for all weights $\mu$ of $V$},\\
Z(G_V)
&=\{h_{\alpha_1^{\vee}}(g_1)\cdots h_{\alpha_n^{\vee}}(g_n)\mid g_1^{\langle\beta,\alpha_1^{\vee}\rangle}\cdots g_n^{\langle\beta,\alpha_n^{\vee}\rangle}=1\quad\textrm{for all $\beta\in R$}\},
\end{align*}
and if $\FF$ is big enough
$$
T_V=\{h_{\omega_1^{\vee}}(g_1)\cdots h_{\omega_n^{\vee}}(g_n)\mid g_1,\ldots,g_n\in\FF^{\times}\},
$$
where $\omega_1^{\vee},\ldots,\omega_n^{\vee}$ is a $\ZZ$-basis of the $\ZZ$-span of the weights of $V$ \cite[Lemma 35]{St}.
\end{remark}

\section{Labeling points of the flag variety $G/B$}

In this section we follow \cite[Ch.\,8]{St} to show that the points of the flag variety
are naturally indexed by labeled walks.  This is the first step in making a precise connection
between the points in the flag variety and the alcove walk theory in \cite{Ra}.

Let $G$ be a Tits group as in \eqref{Titsgroup} over the field $\FF=\CC$.  
The \textit{root subgroups}
\begin{equation}\label{rootsubgroups}
\cX_{\alpha}=\{x_{\alpha}(c)\mid c\in\CC\},\ \  \textrm{for $\alpha\in R_{\mathrm{re}}$,}
\qquad\hbox{satisfy}\qquad
w\cX_{\beta}w^{-1}=\cX_{w\beta},
\end{equation}
for $w\in W$ and $\beta\in R_{\mathrm{re}}$,
since $h_{\alpha^{\vee}}(c)\cX_{\beta}h_{\alpha^{\vee}}(c)^{-1}=\cX_{\beta}$ and $n_{\alpha}\cX_{\beta}n_{\alpha}^{-1}=\cX_{s_{\alpha}\beta}$.
As a group $\cX_{\alpha}$ is isomorphic to $\CC$ (under addition). 

The \textit{flag variety} is $G/B$, where the subgroup
\begin{equation}\label{Borel}
\hbox{$B$ is generated by $T$ and $x_{\alpha}(f)$ for $\alpha\in R_{\mathrm{re}}^+$, 
$f\in\CC$.}
\end{equation}
Let $w\in W$. The \textit{inversion set of $w$} is
\begin{equation}\label{invset}
R(w)=\{\alpha\in R_{\mathrm{re}}^+\mid w^{-1}\alpha\notin R_{\mathrm{re}}^+\}
\qquad\textrm{and}\qquad 
\ell(w)=\Card(R(w))
\end{equation}
is the \textit{length of $w$}. View a reduced expression 
$\vec w=s_{i_1}\cdots s_{i_{\ell}}$ in the generators in \eqref{Wpres} as
a \emph{walk} in $W$ starting at $1$ and ending at $w$,  
\begin{equation}\label{walkseq}
1\quad\longrightarrow\quad s_{i_1}\quad\longrightarrow\quad s_{i_1}s_{i_2}\quad\longrightarrow\quad\cdots\quad\longrightarrow\quad s_{i_1}\cdots s_{i_{\ell}}=w.
\end{equation}
Letting $x_i(c) =x_{\alpha_i}(c)$ and $n_i= n_{\alpha_i}(1)$, 
the following theorem shows that 
\begin{equation}\label{stepform}
BwB = \{ x_{i_1}(c_1)n_{i_1}^{-1}x_{i_2}(c_2)n_{i_2}^{-1}\cdots
x_{i_\ell}(c_\ell)n_{i_\ell}^{-1}B\ |\ c_1,\ldots, c_\ell\in \CC\}
\end{equation}
so that the $G/B$-points of $BwB$ are in bijection with labelings of the 
edges of the walk by complex numbers $c_1,\ldots, c_\ell$.
The elements of $R(w)$ are 
\begin{equation}\label{rootseq}
\beta_1=\alpha_{i_1},\quad \beta_2=s_{i_1}\alpha_{i_2},\quad\ldots,\quad \beta_{\ell}=s_{i_1}\cdots s_{i_{\ell-1}}\alpha_{i_{\ell}},
\end{equation}
and the first relation in \eqref{eq:R4} gives
\begin{equation}
x_{i_1}(c_1)n_{i_1}^{-1}x_{i_2}(c_2)n_{i_2}^{-1}\cdots
x_{i_\ell}(c_\ell)n_{i_\ell}^{-1}
=x_{\beta_1}(\pm c_1)\cdots x_{\beta_\ell}(\pm c_\ell)n_w,
\end{equation}
where $n_w = n_{i_1}^{-1}\cdots n_{i_\ell}^{-1}$.

\begin{thm}\label{thm:BwB} \emph{\cite[Thm.\,15 and Lemma 43]{St}} Let $w\in W$ 
and let $n_w$ be a representative of~$w$ in~$N$.
If
$$R(w) = \{ \beta_1,\ldots, \beta_\ell\}
\qquad\hbox{then}\qquad
\{x_{\beta_1}(c_1)\cdots x_{\beta_{\ell}}(c_{\ell})n_w\mid c_1,\ldots,c_\ell\in
\CC\}$$
is a set of representatives of the $B$-cosets in
$BwB$.
\end{thm}
\begin{proof}
The conceptual reason for this is that 
\begin{align*}
BwB &= \left(\prod_{\alpha\in R^+_{\mathrm{re}}} \cX_\alpha\right)n_w B
= n_w 
\left(\prod_{w^{-1}\alpha\not\in R^+_{\mathrm{re}}} \cX_{w^{-1}\alpha}\right)
\left(\prod_{w^{-1}\alpha\in R^+_{\mathrm{re}}} \cX_{w^{-1}\alpha}\right)
B  \\
&= n_w 
\left(\prod_{w^{-1}\alpha\not\in R^+_{\mathrm{re}}} \cX_{w^{-1}\alpha}\right)
B 
=\left(\prod_{\alpha\in R(w)} \cX_{\alpha}\right) n_wB \\
&= \{ x_{\beta_1}(c_1)\cdots x_{\beta_\ell}(c_\ell)n_wB\ |\ c_1,\ldots, c_\ell\in \FF \}.
\end{align*}
Since $R_{\mathrm{re}}^+$ may be infinite there is a subtlety in the decomposition
and ordering of the product  of $\cX_\alpha$ in the second ``equality'' and it is necessary
to proceed more carefully. Choose a reduced decomposition $w=s_{i_1}\cdots s_{i_{\ell}}$ and let $\beta_1,\ldots,\beta_{\ell}$ be the ordering of $R(w)$ from (\ref{rootseq}).

\smallskip\noindent
Step 1: Since $R(w)\subseteq R^+_{\mathrm{re}}$  there is an inclusion
$$\{x_{\beta_1}(c_1)\cdots x_{\beta_{\ell}}(c_{\ell})n_wB\mid c_1,\ldots, c_\ell\in
\CC\}\subseteq BwB.$$ 
To prove equality proceed by induction on $\ell$.  

\emph{Base case:} Suppose that $w=s_j$. 
Let $\alpha\in R_{\mathrm{re}}^+$ and $c,d\in \CC$. 
If $c=0$ or $\alpha,\alpha_j$ is a prenilpotent pair then, by relation \eqref{commutatorreln}, 
\begin{align}\label{eq:Binv}
x_{\alpha}(d)x_{\alpha_j}(c)n_j^{-1}B=x_{\alpha_j}(c')n_j^{-1}B,
\qquad\textrm{for some $c'\in\CC$}.
\end{align}
If $\alpha,\alpha_j$ is not a prenilpotent pair and $c\neq 0$ then 
$\alpha,-\alpha_j$ is a prenilpotent pair and, by \eqref{commutatorreln},
$$
x_{\alpha}(d)x_{\alpha_j}(c)n_j^{-1}B
=x_{\alpha}(d)x_{-\alpha_j}(c^{-1})B
=x_{-\alpha_j}(c^{-1})B=x_{\alpha_j}(c)n_j^{-1}B.
$$
Thus $\{x_{\alpha_j}(c)n_j^{-1}B\mid c\in\CC\}$ is $B$-invariant
and so $Bs_jB=\{x_{\alpha_j}(c)n_j^{-1}B\mid c\in\CC\}$.

\emph{Induction step:} If $w=s_{i_1}\cdots s_{i_\ell}$ is reduced
and if $\ell(ws_j)>\ell(w)$ then, by induction, 
\begin{align*}
Bws_jB
&\subseteq BwB\cdot Bs_jB
=\{ x_{\beta_1}(c_1)\cdots x_{\beta_{\ell}}(c_{\ell})x_{w\alpha_j}(c)n_wn_j^{-1}B
\mid c_1,\ldots, c_{\ell}, c \in \FF \},
\end{align*}
so that 
$
Bws_jB=\{x_{\beta_1}(c_1)\cdots x_{\beta_{\ell+1}}(c_{\ell+1})n_{ws_j}B
\mid c_1,\ldots,c_{\ell+1}\in\CC\}$ with $\beta_{\ell+1}=w\alpha_j$.

\smallskip\noindent
Step 2: Prove that \quad $BwB=BvB$ if and only if $w=v$\quad 
by induction on $\ell(w)$.

\smallskip
\emph{Base case:}
Suppose that $\ell(w)=0$.  Then $BwB=BvB$ implies that $v\in B$ so that there is a 
representative $n_v$ of $v$ such that $n_v\in B\cap N$.
Then $vR_{\mathrm{re}}^+\subseteq R_{\mathrm{re}}^+$ since
$n_v\cX_{\alpha}n_v^{-1}=\cX_{v\alpha}\in B$ for $\alpha\in R_{\mathrm{re}}^+$.  So
$\ell(v)=0$.  Thus, by \eqref{Wpres}, $v=1$.

\smallskip
\emph{Induction step:} Assume $BwB=BvB$ and $s_j$ is such that $\ell(ws_j)<\ell(w)$. 
Since $BvB\cdot Bs_jB \subseteq BvB\cup Bvs_jB$ (see \cite[Lemma 25]{St}, 
$$
Bws_jB\subseteq BwB\cdot Bs_jB=BvB\cdot Bs_jB\subseteq BvB\cup
Bvs_jB=BwB\cup Bvs_jB.$$
Thus, by induction, $ws_j=w$ or
$ws_j=vs_j$. Since $ws_j\ne w$, it follows that $w=v$.

\smallskip\noindent
Step 3:  Let us show that if  
$x_{\alpha_{i_1}}(c_1)n_{i_1}^{-1}\cdots 
x_{\alpha_{i_{\ell}}}(c_{\ell})n_{i_{\ell}}^{-1}B
=x_{\alpha_{i_1}}(c_1')n_{i_1}^{-1}\cdots x_{\alpha_{i_{\ell}}}(c_{\ell}')n_{i_{\ell}}^{-1}B,
$
then $c_i = c_i'$ for $i=1,2,\ldots, \ell$.  The left hand side of
$$
x_{\alpha_2}(c_2)n_{i_2}^{-1}\cdots x_{i_\ell}(c_\ell)n_{i_{\ell}}^{-1}B
=n_{i_1}x_{i_1}(c_1'-c_1)n_{i_1}^{-1}\cdots x_{i_\ell}(c_{\ell}')n_{i_\ell}^{-1}B
$$
is in $Bs_{i_2}\cdots s_{i_\ell}B$. If
$c_1'\neq c_1$ then $n_{i_1}^{-1}x_{i_1}(c_1'-c_1)n_{i_1}\in Bs_{i_1}B$
and the right hand side is contained in
$$
n_{i_1}^{-1}x_{i_1}(c_1'-c_1)n_{i_1}Bs_{i_2}\cdots s_{i_\ell}B\subseteq 
Bs_{i_1}B\cdot Bs_{i_2}\cdots s_{i_\ell}B
=Bs_{i_1}\cdots s_{i_\ell}B.
$$
By Step 2 this is impossible and so
$c_1'=c_1$.  Then, by induction, $c_i'=c_i$ for $i=1,2,\ldots, \ell$.

\smallskip\noindent
Step 4:  From the definition of $R(w)$ it follows that if $\alpha,\beta\in R(w)$ and 
$\alpha+\beta\in R_{\mathrm{re}}$ then $\alpha+\beta\in R(w)$
and if $\alpha,\beta\in R(w)$ then $\alpha,\beta$ form a prenilpotent pair. 
Thus, by [St, Lemma 17], any total order on the set $R(w)$ can be taken in the statement of 
the theorem. 
\end{proof}

\begin{remark}  Suppose that $\lambda\in \fh^*$ is dominant integral and 
$M(\lambda)$ is an (integrable) highest weight 
representation of $G$ generated by a highest weight vector $v_\lambda^+$. Then the set 
$BwBv_\lambda^+$ contains the vector $wv_\lambda^+$ and is
contained in the sum $\bigoplus_{\nu\ge w\lambda} M(\lambda)_\nu$ of the weight spaces
with weights $\ge w\lambda$.  This is another way to show that if $w\ne v$ then $BwB\ne BvB$
and accomplish Step~2 in the proof of Theorem \ref{thm:BwB}. 
\end{remark}

\section{Loop Lie algebras and their extensions}

This section gives a presentation of the theory of loop Lie algebras.  The main
lines of the theory are exactly as in the classical case (see, for example,
\cite[\S 4]{Mac3} and \cite[ch.\,7]{Kac}) but, following recent trends
(see \cite{Ga2}. \cite{GK}, \cite{GR} and \cite{Rou}) we treat the more general setting of the
loop Lie algebra of a Kac-Moody Lie algebra.

Let $\fg_0$ be a symmetrizable Kac-Moody Lie algebra with 
bracket $[,]_0\colon \fg_0\otimes \fg_0 \to \fg_0$ and invariant
form $\langle, \rangle_0 \colon \fg_0\times \fg_0 \to \CC$.
The \emph{loop Lie algebra} is 
$$\fg_0[t,t^{-1}] = \CC[t,t^{-1}]\otimes_\CC \fg_0
\qquad\hbox{with bracket}\qquad
[t^m x, t^n y]_0 = t^{m+n}[x,y]_0,$$
for $x,y\in \fg_0$.
Let
\begin{align*}
\fg = \fg_0[t,t^{-1}] \oplus \CC c \oplus \CC d, \qquad
\fg' = \fg_0[t, t^{-1}] \oplus \CC c, \qquad
\bar \fg' = \fg_0[t,t^{-1}] = \frac{\fg'}{\CC c}
\end{align*}
where the bracket on $\fg$ is given by 
\begin{equation}\label{loopliedefn}
[t^mx, t^ny] = t^{m+n}[x,y]_0 + \delta_{m+n,0} m \langle x,y\rangle_0 c,
\qquad
c\in Z(\fg),\qquad [d,t^mx] = mt^m x.
\end{equation}
By \cite[Ex.\ 7.8]{Kac}, $\fg'$ is the universal central extension of $\bar\fg'$.
An invariant symmetric form on $\fg$ is given by 
\begin{equation}
\langle c,d\rangle = 1,
\qquad 
\langle c, t^my\rangle = \langle d, t^m y\rangle = 0,
\qquad
\langle c, c\rangle = \langle d, d\rangle = 0,
\end{equation}
and
\begin{equation}
\langle t^mx, t^ny\rangle
= \begin{cases}
\langle x, y\rangle_0, &\hbox{if $m+n=0$,} \\
0, &\hbox{otherwise,}
\end{cases}
\end{equation}
for $x,y\in \fg_0$, $m,n\in \ZZ$.

Fix a Cartan subalgebra $\fh_0$ of $\fg_0$ and let
\begin{equation}
\fh = \fh_0 \oplus \CC c \oplus \CC d,
\qquad
\fh' = \fh_0 \oplus \CC c,
\qquad
\bar \fh' = \fh_0.
\end{equation}
As in \eqref{fhdef}, let $h_1,\ldots, h_n, d_1,\ldots, d_\ell$ be a basis of $\fh_0$ and let
\begin{equation}
\begin{array}{l}
\hbox{$\{h_1,\ldots, h_n, d_1,\ldots, d_\ell, c,d\}$ be a basis of $\fh$ and} \\
\hbox{$\{\omega_1, \ldots, \omega_n, \delta_1,\ldots, \delta_\ell, \Lambda_0, \delta\}$ 
the dual basis in 
$\fh^*$}
\end{array}
\end{equation}
so that
\begin{equation}\label{deltadefn}
\delta(\fh_0)=0,\ \ \delta(c)=0,\ \ \delta(d)=1,\qquad
\hbox{and}\qquad
\Lambda_0(\fh_0)=0,\ \ \Lambda_0(c)=1,\ \ \Lambda_0(d)=0.
\end{equation}
Let $R$ be as in \eqref{rootdef}.
As an $\fh$-module
\begin{equation}
\fg = 
\left( \bigoplus_{\alpha\in R\atop k\in \ZZ} \fg_{\alpha+k\delta}\right)
\oplus
\left( \bigoplus_{k\in \ZZ_{\ne 0}} \fg_{k\delta}\right)
\oplus\fh,
\quad\hbox{where}\quad
\fh = \fh_0 \oplus \CC c \oplus \CC d,
\end{equation}
\begin{equation}
\fg_{\alpha+k\delta} = t^k \fg_\alpha,
\qquad
\fg_{k\delta} = t^k \fh_0,
\qquad\hbox{and}\qquad
\tilde R = (R + \ZZ\delta) \cup \ZZ_{\ne 0}\delta
\end{equation}
is the set of \textit{roots} of $\fg$.

Let $\alpha\in R_{\mathrm{re}}$ with $\alpha = w\alpha_i$ and fix a choice of
$e_\alpha, f_\alpha$ and $h_\alpha$ in \eqref{alphasl2} (choose $\tilde w$).  Then
\begin{align}\label{eq:loopgens}
e_{-\alpha+k\delta} = t^kf_\alpha,
\qquad
f_{-\alpha+k\delta} =  t^{-k} e_\alpha, \qquad
h_{-\alpha+k\delta} 
%= [e_{-\alpha+k\delta}, f_{-\alpha+k\delta}]
= -h_\alpha + k  \langle e_\alpha, f_\alpha\rangle_0 c,
\end{align}
span a subalgebra isomorphic to $\fsl_2$.
If $\fg_0 = \fn_0^-\oplus \fh_0 \oplus \fn_0^+$ is the decomposition in 
\eqref{KMtridecomp} and
$$\begin{array}{l}
\hbox{$\fn^+$ is the subalgebra generated by $\fn_0^+$ and
$e_{-\alpha+k\delta}$ for $\alpha\in R_{\mathrm{re}}$, 
$k\in \ZZ_{>0}$, and} \\
\hbox{$\fn^-$ is the subalgebra generated by $\fn_0^-$ and 
$f_{-\alpha+k\delta}$ for $\alpha\in R_{\mathrm{re}}$, 
$k\in \ZZ_{>0}$,} \\
\end{array}
$$
then
$$\fg = \fn^- \oplus \fh \oplus \fn^+
\quad\hbox{with}\quad
\fn^+ = \fn_0^+\oplus 
\left(\bigoplus_{\alpha\in R\cup\{0\}\atop k\in \ZZ_{>0}} \fg_{\alpha+k\delta}\right)
%\oplus
%\left( \bigoplus_{k\in \ZZ_{>0}} \fg_{k\delta}\right)
\quad\hbox{and}\quad
\fn^- = \fn_0^-\oplus 
\left(\bigoplus_{\alpha\in R\cup\{0\}\atop k\in \ZZ_{<0}} \fg_{\alpha+k\delta}\right).
%\oplus
%\left( \bigoplus_{k\in \ZZ_{<0}} \fg_{k\delta}\right).
$$
%Is there a $\ZZ$-grading here???  What are the generators of this $\ZZ$-grading???

The elements $e_{-\alpha+k\delta}$ 
and $f_{-\alpha+k\delta}$ in \eqref{eq:loopgens}
act locally nilpotently on $\fg$ because $f_\alpha$ and 
$e_\alpha$ act locally nilpotently on $\fg_0$.
Thus
\begin{equation}
\tilde s_{-\alpha+k\delta}
= \exp(\ad\,{t^kf_\alpha})\exp(-\ad\,{t^{-k}e_\alpha})\exp(\ad\,{t^kf_\alpha})
\end{equation}
is a well defined automorphism of $\fg$ and 
\begin{equation}
\tilde s_{-\alpha+k\delta} \fg_\beta = \fg_{s_{-\alpha+k\delta}\beta}
\qquad\hbox{and}\qquad
\tilde s_{-\alpha+k\delta}h = s_{-\alpha+k\delta} h,
\end{equation}
for $h\in \fh$ and $\beta\in \tilde R$, where\quad
$s_{-\alpha+k\delta}\colon \fh^*\to \fh^*$\quad and\quad $s_{-\alpha+k\delta}\colon \fh\to \fh$
\quad
are given by
\begin{equation}\label{affrefl}
s_{-\alpha+k\delta}\lambda = \lambda - \lambda(h_{-\alpha+k\delta})(-\alpha+k\delta)
\quad\hbox{and}\quad
s_{-\alpha+k\delta} h = h - (-\alpha+k\delta)(h)h_{-\alpha+k\delta},
\end{equation}
for $\lambda\in \fh^*$ and $h\in \fh$.
The \emph{Weyl group} of $\fg$ is the subgroup of $GL(\fh^*)$ (or $GL(\fh)$)
generated by the reflections $s_{-\alpha+k\delta}$,
\begin{equation}
W_{\mathrm{aff}} = \langle s_{-\alpha+k\delta}\ |\ \alpha\in R_{\mathrm{re}}, k\in \ZZ\rangle.
\end{equation}

Noting that \quad
$\fh^* = \fh_0^*\oplus\CC \Lambda_0 \oplus \CC\delta$
\quad and \quad
$\fh = \fh_0 \oplus \CC c \oplus \CC d$,
use \eqref{affrefl} to compute
\begin{equation*}
\begin{array}{ll}
s_{-\alpha+k\delta}(\bar\lambda) = \bar\lambda+\bar\lambda(h_\alpha)(-\alpha+k\delta), \qquad
&s_{-\alpha+k\delta}(\bar h) = \bar h + \alpha(\bar h)
(-h_\alpha+k\langle e_\alpha,f_\alpha\rangle_0 c), 
\\
s_{-\alpha+k\delta}(\ell\Lambda_0) = \ell\Lambda_0 - k\ell \langle e_\alpha, f_\alpha\rangle_0
(-\alpha+k\delta),
&s_{-\alpha+k\delta}(m c) = m c, \\
s_{-\alpha+k\delta}(m\delta) = m\delta,
&s_{-\alpha+k\delta}(\ell d) = \ell d - k\ell (-h_\alpha+k\langle e_\alpha,f_\alpha\rangle_0 c).
\end{array}
\end{equation*}
for $\bar\lambda\in \fh_0^*$, $\bar h\in \fh_0$, $m,\ell\in \CC$. 
For $\alpha\in R_\mathrm{re}$ and $k\in \ZZ$ 
\begin{equation}\label{roottransl}
\hbox{define $t_{k\alpha^\vee}\in W_{\mathrm{aff}}$ by}\qquad
s_{-\alpha+k\delta} = t_{k\alpha^\vee}s_{-\alpha},
\end{equation}
and use \eqref{halphavee} and \eqref{alphavee} 
to compute
\begin{equation*}
\begin{array}{ll}
t_{k\alpha^\vee}(\bar\lambda) = \bar\lambda-\bar\lambda(kh_\alpha)\delta, \qquad
&t_{k\alpha^\vee}(\bar h) = \bar h - k\alpha^\vee (\bar h) c, 
\\
t_{k\alpha^\vee}(\ell\Lambda_0) 
= \ell\Lambda_0 +\ell k\alpha^\vee- \ell
\hbox{$\frac12$}\langle kh_\alpha, kh_\alpha\rangle_0\delta,
&t_{k\alpha^\vee}(m c) = m c, \\
t_{k\alpha^\vee}(m\delta) = m \delta,
&t_{k\alpha^\vee}(\ell d) = \ell d +\ell kh_\alpha 
- \ell \hbox{$\frac12$} \langle kh_\alpha, kh_\alpha\rangle_0 c.
\end{array}
\end{equation*}
Then
$t_{k\alpha^\vee}t_{j \beta^\vee}(\bar\lambda)
=t_{kh_\alpha}(\bar\lambda - \bar\lambda(jh_\beta)\delta)
=\bar\lambda -\bar\lambda(kh_\alpha+j h_\beta)\delta$, and
\begin{align*}
t_{k\alpha^\vee}t_{j\beta^\vee}(\ell\Lambda_0)
&=t_{k\alpha^\vee}\left( \ell\Lambda_0 +\ell j\beta^\vee - \ell\hbox{$\frac12$}
\langle jh_\beta, jh_\beta\rangle_0\delta \right) \\
&= \ell\Lambda_0 + \ell k\alpha^\vee - \ell \hbox{$\frac12$}
\langle kh_\alpha, k h_\alpha\rangle_0\delta
+\ell j \beta^\vee - \ell j \beta^\vee(kh_\alpha)\delta 
- \ell \hbox{$\frac12$}\langle jh_\beta, jh_\beta\rangle_0 \delta \\
&= \ell \Lambda_0 +\ell(k\alpha^\vee + j\beta^\vee)
-\ell \hbox{$\frac12$}\langle kh_\alpha+jh_\beta, kh_\alpha+jh_\beta\rangle_0 \delta.
\end{align*}
This computation shows that 
$t_{k\alpha^\vee}t_{j\beta^\vee}=t_{j\alpha^\vee+k\beta^\vee}$.
%\quad\hbox{BUT ARE}\quad
%h_{k\alpha+j\beta} = kh_\alpha+ jh_\beta
%\ \ \hbox{AND}\ \ 
%(k\alpha+j\beta)^\vee = k\alpha^\vee + j\beta^\vee?????
%$$
Thus, if $W_0$ is the Weyl group of $\fg_0$ and 
$Q^* = \ZZ\hbox{-span}\{\alpha_1^\vee,\ldots, \alpha_n^\vee\}$
then
\begin{equation}\label{Wsdprod}
W_{\mathrm{aff}} = \{ t_{\lambda^\vee}w\ |\ \lambda^\vee\in Q^*, w\in W_0\}
\qquad\hbox{with}\qquad
t_{\lambda^\vee}t_{\mu^\vee}= t_{\lambda^\vee+\mu^\vee}
\quad\hbox{and}\quad
wt_{\lambda^\vee} = t_{w\lambda^\vee} w,
\end{equation}
for $w\in W_0$, $\lambda^\vee,\mu^\vee\in Q^*$.

Since $\CC \delta$ is $W_{\mathrm{aff}}$-invariant, the group
$W_{\mathrm{aff}}$ acts on $\fh^*/\CC\delta$ and $W_{\mathrm{aff}}$ acts on the
set
\begin{equation}\label{levelone}
\begin{matrix}
(\fh_0^*+\Lambda_0+\CC\delta)/\CC\delta &\mapright{\sim} &\fh_0^* \\
\bar\lambda+\Lambda_0 + \CC\delta &\longmapsto &\bar \lambda
\end{matrix}
\end{equation}
and the $W_{\mathrm{aff}}$-action on the right hand side is given by
\begin{equation}\label{nonlinearW}
s_\alpha(\bar\lambda) = \bar\lambda - \bar\lambda(h_\alpha)\alpha
\qquad\hbox{and}\qquad
t_{k\alpha^\vee}(\bar\lambda) = \bar\lambda+ k\alpha^\vee,\qquad
\hbox{for $\bar\lambda\in \fh_0$.}
\end{equation}
Here $\fh_0^*$ is a set with a $W_{\mathrm{aff}}$-action, the action of $W_{\mathrm{aff}}$ is
\emph{not linear}.

\section{Loop groups and the affine flag variety $G/I$}

This section gives a short treatment of loop groups following \cite[Ch.\,8]{St} and
\cite[\S 2.5 and 2.6]{Mac1}.  This theory is currently a subject of intense research as evidenced
by the work in \cite{Ga2}, \cite{GK}, \cite{Rem}, \cite{Rou}, \cite{GR}.

Let $\fg_0$ be a symmetrizable Kac-Moody Lie algebra and let $\fh_\ZZ$ be a
$\ZZ$-lattice in $\fh_0$ that contains $Q^\vee = \ZZ\hbox{-span}\{h_1,\ldots, h_n\}$.
\begin{equation}
\hbox{The \emph{loop group} is the Tits group $G=G_0(\CC((t)))$}
\end{equation}
over the field $\FF = \CC((t))$.
Let $K=G_0(\CC[[t]])$ and $G_0(\CC)$ be the Tits groups of $\fg_0$ and $\fh_\ZZ$ 
over the rings $\CC[[t]]$ and $\CC$, respectively, and let $B(\CC)$ be the 
standard \emph{Borel subgroup} of $G_0(\CC)$ as defined in \eqref{Borel}. Let
\begin{equation}
\hbox{$U^-$ be the subgroup of $G$ generated by $x_{-\alpha}(f)$ for $\alpha\in R_{\mathrm{re}}^+$ and $f\in \CC((t))$,}
\end{equation}
and define the standard
\emph{Iwahori subgroup} $I$ of $G$ by 
\begin{equation}\label{Idefn}
\begin{matrix}
G &= &G_0(\CC((t))) \\
$\beginpicture
\setcoordinatesystem units <0.8cm,0.8cm>         % sets scale
\setplotarea x from -0.1 to 0.3, y from -0.5 to 0.5  % sets plot size up
\put{$\cup$} at 0 0 \put{$\shortmid$} at 0.21 0.08
\put{$\shortmid$} at 0.21 0.0
\endpicture$ &&$\beginpicture
\setcoordinatesystem units <0.8cm,0.8cm>         % sets scale
\setplotarea x from -0.1 to 0.3, y from -0.5 to 0.5  % sets plot size up
\put{$\cup$} at 0 0 \put{$\shortmid$} at 0.21 0.08
\put{$\shortmid$} at 0.21 0.0
\endpicture$ \\
K &=& G_0(\CC[[t]]) &\mapright{\mathrm{ev}_{t=0}} &G_0(\CC) \\
$\beginpicture
\setcoordinatesystem units <0.8cm,0.8cm>         % sets scale
\setplotarea x from -0.1 to 0.3, y from -0.5 to 0.5  % sets plot size up
\put{$\cup$} at 0 0 \put{$\shortmid$} at 0.21 0.08
\put{$\shortmid$} at 0.21 0.0
\endpicture$ &&$\beginpicture
\setcoordinatesystem units <0.8cm,0.8cm>         % sets scale
\setplotarea x from -0.1 to 0.3, y from -0.5 to 0.5  % sets plot size up
\put{$\cup$} at 0 0 \put{$\shortmid$} at 0.21 0.08
\put{$\shortmid$} at 0.21 0.0
\endpicture$ &&$\beginpicture
\setcoordinatesystem units <0.8cm,0.8cm>         % sets scale
\setplotarea x from -0.1 to 0.3, y from -0.5 to 0.5  % sets plot size up
\put{$\cup$} at 0 0 \put{$\shortmid$} at 0.21 0.08
\put{$\shortmid$} at 0.21 0.0
\endpicture$ \\
I &= &\mathrm{ev}_{t=0}^{-1}(B(\CC)) &\mapright{\mathrm{ev}_{t=0}} &B(\CC).
\end{matrix}
\end{equation}
The \emph{affine flag variety} is $G/I$.

For $\alpha+j\delta\in R_{\mathrm{re}}+\ZZ\delta$ and $c\in \CC$, define
\begin{equation}\label{affxalpha}
x_{\alpha+j\delta}(c)=x_{\alpha}(ct^j)
\qquad\textrm{and}\qquad 
t_{\la^{\vee}}=h_{\la^{\vee}}(t^{-1}),
\end{equation}
and, for $c\in \CC^\times$, define
\begin{equation}\label{affn}
n_{\alpha+j\delta}(c) = x_{\alpha+j\delta}(c)
x_{-\alpha-j\delta}(-c^{-1})x_{\alpha+j\delta}(c),
\end{equation}
\begin{equation}\label{affh}
n_{\alpha+j\delta} = n_{\alpha+j\delta}(1),
\quad\hbox{and}\quad
h_{(\alpha+j\delta)^\vee}(c) = n_{\alpha+j\delta}(c)n_{\alpha+j\delta}^{-1}
\end{equation}
analogous to \eqref{nalphadefn}.  

The group
\begin{equation}
\widetilde{W} = \{ t_{\lambda^\vee} w\ |\ \lambda^\vee\in \fh_\ZZ, w\in W_0\}
\qquad\hbox{with}\quad
t_{\lambda^\vee} t_{\mu^\vee} = t_{\lambda^\vee+\mu^\vee}
\quad\hbox{and}\quad wt_{\lambda^\vee} = t_{w\lambda^\vee} w,
\end{equation}
acts on $\fh_0^*\oplus \CC\delta$ by 
\begin{align}\label{eq:action}
v(\mu+k\delta)=v\mu+k\delta
\qquad\textrm{and}\qquad 
t_{\la^{\vee}}(\mu+k\delta)=\mu+(k-\langle\la^{\vee},\mu\rangle)\delta
\end{align}
for $v\in W_0$, $\la^{\vee}\in \fh_\ZZ$, $\mu\in\fh_\ZZ^*$, and $k\in\ZZ$. 
Then
$
n_{\alpha+j\delta}(c)=t_{-j\alpha^{\vee}}n_{\alpha}(c)=n_\alpha(ct^j),
$
$$
n_{\alpha}x_{\beta+k\delta}(c)n_{\alpha}^{-1}
=n_{\alpha}x_{\beta}(ct^k)n_{\alpha}^{-1}
=x_{s_{\alpha}\beta}(\epsilon_{\alpha,\beta} ct^k)
=x_{s_{\alpha}(\beta+k\delta)}(\epsilon_{\alpha,\beta} c)
$$
for $\alpha\in R_{\mathrm{re}}$,
and, for $\lambda^\vee\in\fh_\ZZ$,
$$
t_{\la^{\vee}}x_{\beta+k\delta}(c)t_{\la^{\vee}}^{-1}=x_{\beta+k\delta}(t^{-\langle\la^{\vee},\beta\rangle}c)=x_{t_{\la^{\vee}}(\beta+k\delta)}(c).
$$
Thus the \emph{root subgroups}
\begin{equation}
\cX_{\alpha+j\delta} = \{ x_{\alpha+j\delta}(c)\ |\ c\in \CC\}
\qquad\hbox{satisfy}\qquad
w\cX_{\alpha+j\beta}w^{-1} = \cX_{w(\alpha+j\delta)}
\end{equation}
for $w\in \widetilde{W}$ and $\alpha+j\delta\in R_{\mathrm{re}}+\ZZ\delta$.
These relations are a reflection of the symmetry of the group $G$ under the group
defined in \eqref{Ndefn}:
\begin{equation}
\widetilde{N}= N(\CC((t)))
\quad\hbox{generated by $n_\alpha(g)$, $h_{\lambda^\vee}(g)$, for $g\in \CC((t))^\times$,}
\end{equation}
$\alpha\in R_{\mathrm{re}}$, and $\lambda^\vee\in \fh_\ZZ$.
The homomorphism $\widetilde{N}\to W_0$ from \eqref{NtoW} lifts to a surjective homomorphism
(see \cite[p.\,26 and p.\,28]{Mac1})
$$\begin{matrix}
\widetilde{N} &\longrightarrow &\widetilde{W} \\
n_{\alpha+j\delta} &\longmapsto &t_{-j\alpha^\vee}s_\alpha \\
t_{\lambda^\vee} &\longmapsto &t_\lambda^\vee
\end{matrix}
\qquad\hbox{with kernel $H$ generated by $h_\lambda(d)$, $d\in \CC[[t]]^\times$.}$$

Define
\begin{equation}
\tilde{R}_{\mathrm{re}}^I
=(R_{\mathrm{re}}^++\ZZ_{\geq0}\delta)\sqcup(-R_{\mathrm{re}}^++\ZZ_{>0}\delta)
\quad\hbox{and}\quad
\tilde{R}_{\mathrm{re}}^U
=-R_{\mathrm{re}}^++\ZZ\delta
\end{equation}
so that
\begin{equation}
\begin{array}{lcl}
\cX_{\alpha+j\delta}\subseteq I 
&\quad\hbox{if and only if}\quad
&\alpha+j\delta\in\tilde{R}_{\mathrm{re}}^I
\qquad\hbox{and} \\
\\
\cX_{\alpha+j\delta}\subseteq U^-
&\quad\hbox{if and only if}\quad
&\alpha+j\delta\in\tilde{R}_{\mathrm{re}}^U.
\end{array}
\end{equation}
Note that 
$\tilde{R}_{\mathrm{re}}^I\sqcup(-\tilde{R}_{\mathrm{re}}^I)
=\tilde{R}_{\mathrm{re}}^U\sqcup(-\tilde{R}_{\mathrm{re}}^U)
=R_{\mathrm{re}}+\ZZ\delta$.

\section{The folding algorithm and the intersections $U^-vI\cap IwI$}

In this section we prove our main theorem, which gives a precise connection
between the alcove walks in \cite{Ra} and the points in the affine flag variety.
The algorithm here is essentially that which is found in \cite{BD} and, with our setup from
the earlier sections, it is the `obvious one'.  The same method has, of course,
been used in other contexts, see, for example, \cite{C}.

A special situation in the loop group theory is when $\fg_0$ is finite dimensional.  
In this case, the extended loop Lie algebra $\fg$ defined in \eqref{loopliedefn} 
is also a Kac-Moody Lie algebra.
If $G_0$ is the Tits group of $\fg_0$ and $G=G_0(\CC((t)))$ is the corresponding
loop group then the subgroup $I$ defined in \eqref{Idefn} differs from the Borel subgroup of 
the Kac-Moody group $G_{KM}$ for $\fg$ only by elements of $T$, and 
the affine flag variety of $G$ coincides with the flag variety of $G_{KM}$.  
Thus, in this case, Theorem \ref{thm:BwB} provides a labeling of the points of the affine
flag variety. 

Suppose that $\fg_0$ is a finite dimensional complex semisimple Lie algebra
presented as a Kac-Moody Lie algebra with generators
$e_1,\ldots, e_n, f_1,\ldots, f_n, h_1,\ldots, h_n$ and Cartan matrix
$A = \left( \alpha_i(h_j)\right)_{1\le i,j,\le n}$.
%Assume that the invariant symmetric bilinear form 
%$\langle, \rangle_0\colon \fg_0\otimes\fg_0\to \CC$
%is normalised so that 
%$$\langle \varphi,\varphi\rangle = 2,
%\qquad\hbox{(FIX THIS LINE)where $\varphi$ is the highest root of $R$}
%$$
%(the highest weight of the adjoint representation).  
Let $\varphi$ be the highest
root of $R$ (the highest weight of the adjoint representation), fix
$$e_\varphi\in \fg_{\varphi},
\quad
f_{\varphi}\in \fg_{-\varphi}
\quad\hbox{such that}\quad
\langle e_\varphi, f_{\varphi}\rangle_0 = 1,$$
and let
$$e_0 = e_{-\varphi+\delta} = tf_{\varphi},
\quad
f_0 = f_{-\varphi+\delta} = t^{-1}e_\varphi,
\qquad
h_0 = [e_0,f_0] = [tx_{-\varphi}, t^{-1}x_\varphi] 
=-h_\varphi+c,
$$
as in \eqref{eq:loopgens}.  The magical fact is that,
in this case, $\fg = \fg_0[t,t^{-1}]\oplus \CC c \oplus \CC d$ is a Kac-Moody Lie 
algebra with generators
$e_0,\ldots, e_n$, $f_0,\ldots, f_n$, $h_0,\ldots, h_n, d$ and Cartan matrix
\begin{equation}
A^{(1)} = \left(\alpha_i(h_j)\right)_{0\le i,j\le n},
\qquad\hbox{where\quad 
$\alpha_0 = -\varphi+\delta$
\quad and\quad 
$h_0= -h_\varphi+c$,}
\end{equation}
where $\delta$ is as in \eqref{deltadefn} (see \cite[Thm.\,7.4]{Kac}).

The \emph{alcoves} are the open connected components of
$$\fh_{\RR}\backslash
\bigcup_{-\alpha+j\delta\in\tilde{R}_\mathrm{re}^I} H_{-\alpha+j\delta},
\qquad\hbox{where}\quad
H_{-\alpha+j\delta} = \{ x^\vee\in \fh_\RR\ |\ \langle x^\vee, \alpha\rangle = j\}.$$
Under the map in \eqref{levelone} the chambers $wC$ of the Tits cone $X$ (see \eqref{domchamber} and \eqref{Titscone}) become the alcoves.
Each alcove is a fundamental region for the action of 
$W_{\mathrm{aff}}$ on $\fh_{\RR}$ given by \eqref{nonlinearW} and
$W_{\mathrm{aff}}$ acts simply transitively on the set of alcoves (see \cite[Prop.\,6.6]{Kac}). 
Identify $1\in W_{\mathrm{aff}}$ with the \textit{fundamental alcove}
$$
A_0=\{x^{\vee}\in \fh_{\RR}\mid \langle
x^{\vee},\alpha_i\rangle>0\textrm{ for all $0\leq i\leq n$}\}$$
to make a bijection
$$W_{\mathrm{aff}} \longleftrightarrow \{\hbox{alcoves}\}.
$$
For example, when $\fg_0= \fsl_3$,
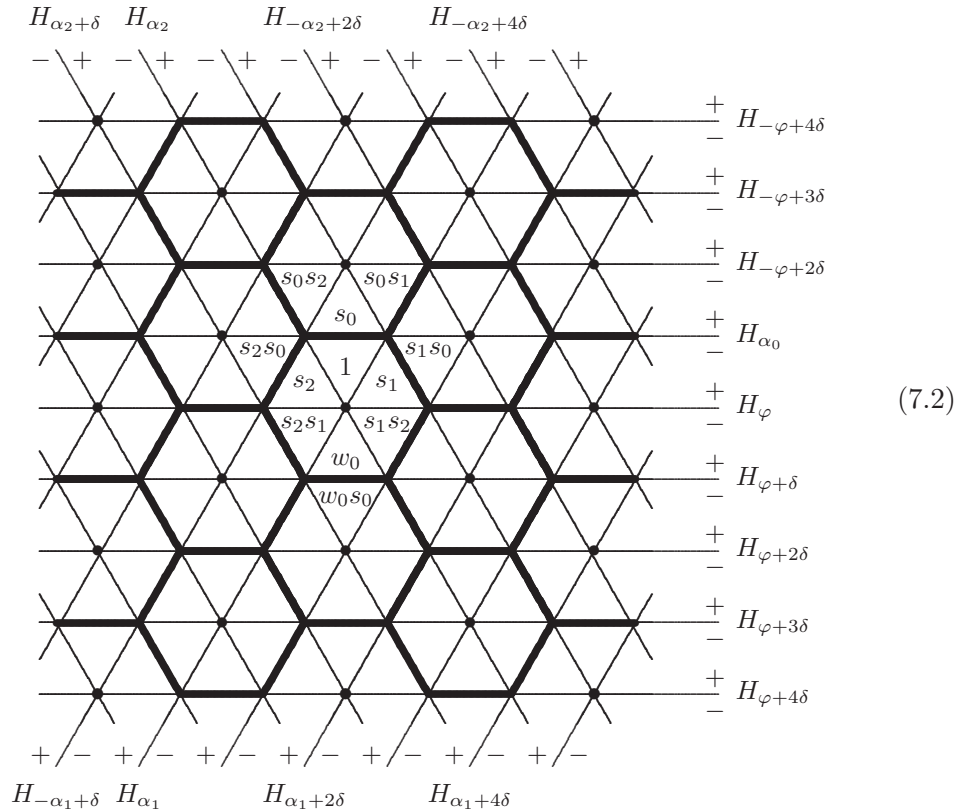
\begin{equation}\label{sl3pict1}
\beginpicture
\setcoordinatesystem units <1.1cm,1.1cm>         % sets scale
\setplotarea x from -5 to 5, y from -5 to 5.2  % sets plot size up
    %slope 1
    \plot 2.8 -3.81 3.7 -2.2516 /
    \plot 1.8 -3.81 3.7 -0.52 /
    \plot 0.8 -3.81 3.7 1.212 /
    \plot -0.2 -3.81 3.7 3.044 /
    \plot -1.2 -3.81 3.2 3.81 /
    \plot -2.2 -3.81 2.2 3.81 /
    \plot -3.2 -3.81 1.2 3.81 /
    \plot -3.7 -3.044 0.2 3.81 /
    \plot -3.7 -1.212 -0.8 3.81 /
    \plot -3.7 0.520 -1.8 3.81 /
    \plot -3.7 2.2516 -2.8 3.81 /
    %slope 2
    \plot -2.8 -3.81 -3.7 -2.2516 /
    \plot -1.8 -3.81 -3.7 -0.52 /
    \plot -0.8 -3.81 -3.7 1.212 /
    \plot 0.2 -3.81 -3.7 3.044 /
    \plot 1.2 -3.81 -3.2 3.81 /
    \plot 2.2 -3.81 -2.2 3.81 /
    \plot 3.2 -3.81 -1.2 3.81 /
    \plot 3.7 -3.044 -0.2 3.81 /
    \plot 3.7 -1.212 0.8 3.81 /
    \plot 3.7 0.520 1.8 3.81 /
    \plot 3.7 2.2516 2.8 3.81 /
    %slope 3
    \plot -3.7 -3.464 3.7 -3.464 /
    \plot -3.7 -2.598 3.7 -2.598 /
    \plot -3.7 -1.732 3.7 -1.732 /
    \plot -3.7 -0.866 3.7 -0.866 /
    \plot -3.7 0 3.7 0 /
    \plot -3.7 3.464 3.7 3.464 /
    \plot -3.7 2.598 3.7 2.598 /
    \plot -3.7 1.732 3.7 1.732 /
    \plot -3.7 0.866 3.7 0.866 /
    %extrabullet
    \put{$\bullet$} at 0 3.464
    \put{$\bullet$} at 0 -3.464
    \put{$\bullet$} at 3 3.464
    \put{$\bullet$} at 3 -3.464
    \put{$\bullet$} at -3 3.464
    \put{$\bullet$} at -3 -3.464
    %gp elementa
    \small{\put{$1$} at 0 0.5
    \put{$s_0$} at 0 1.1
    \put{$s_1$} at 0.5 0.3
    \put{$s_2$} at -0.5 0.3
    \put{$s_0s_1$} at 0.5 1.55
    \put{$s_0s_2$} at -0.5 1.55
    \put{$s_1s_2$} at 0.5 -0.2
    \put{$s_2s_1$} at -0.5 -0.2
    \put{$s_2s_0$} at -1 0.7
    \put{$s_1s_0$} at 1 0.7
    \put{$w_0$} at 0 -0.65
    \put{$w_0s_0$} at 0 -1.1 }%
    %hexagons
    \put{\beginpicture
    \setplotsymbol({\tiny{$\bullet$}})
    \plot 1 0 0.5 0.866 /
    \plot 0.5 0.866 -0.5 0.866 /
    \plot -0.5 0.866 -1 0 /
    \plot -1 0 -0.5 -0.866 /
    \plot -0.5 -0.866 0.5 -0.866 /
    \plot 0.5 -0.866 1 0 /
    \put{$\bullet$} at 0 0
    \endpicture} at 0 0
    \put{\beginpicture
    \setplotsymbol({\tiny{$\bullet$}})
    \plot 1 0 0.5 0.866 /
    \plot 0.5 0.866 -0.5 0.866 /
    \plot -0.5 0.866 -1 0 /
    \plot -1 0 -0.5 -0.866 /
    \plot -0.5 -0.866 0.5 -0.866 /
    \plot 0.5 -0.866 1 0 /
    \put{$\bullet$} at 0 0
    \endpicture} at 0 1.732
    \put{\beginpicture
    \setplotsymbol({\tiny{$\bullet$}})
    \plot 1 0 0.5 0.866 /
    \plot 0.5 0.866 -0.5 0.866 /
    \plot -0.5 0.866 -1 0 /
    \plot -1 0 -0.5 -0.866 /
    \plot -0.5 -0.866 0.5 -0.866 /
    \plot 0.5 -0.866 1 0 /
    \put{$\bullet$} at 0 0
    \endpicture} at 0 -1.732
    \put{\beginpicture
    \setplotsymbol({\tiny{$\bullet$}})
    \plot 1 0 0.5 0.866 /
    \plot 0.5 0.866 -0.5 0.866 /
    \plot -0.5 0.866 -1 0 /
    \plot -1 0 -0.5 -0.866 /
    \plot -0.5 -0.866 0.5 -0.866 /
    \plot 0.5 -0.866 1 0 /
    \put{$\bullet$} at 0 0
    \endpicture} at -1.5 0.866
    \put{\beginpicture
    \setplotsymbol({\tiny{$\bullet$}})
    \plot 1 0 0.5 0.866 /
    \plot 0.5 0.866 -0.5 0.866 /
    \plot -0.5 0.866 -1 0 /
    \plot -1 0 -0.5 -0.866 /
    \plot -0.5 -0.866 0.5 -0.866 /
    \plot 0.5 -0.866 1 0 /
    \put{$\bullet$} at 0 0
    \endpicture} at -1.5 -0.866
    \put{\beginpicture
    \setplotsymbol({\tiny{$\bullet$}})
    \plot 1 0 0.5 0.866 /
    \plot 0.5 0.866 -0.5 0.866 /
    \plot -0.5 0.866 -1 0 /
    \plot -1 0 -0.5 -0.866 /
    \plot -0.5 -0.866 0.5 -0.866 /
    \plot 0.5 -0.866 1 0 /
    \put{$\bullet$} at 0 0
    \endpicture} at -1.5 2.598
    \put{\beginpicture
    \setplotsymbol({\tiny{$\bullet$}})
    \plot 1 0 0.5 0.866 /
    \plot 0.5 0.866 -0.5 0.866 /
    \plot -0.5 0.866 -1 0 /
    \plot -1 0 -0.5 -0.866 /
    \plot -0.5 -0.866 0.5 -0.866 /
    \plot 0.5 -0.866 1 0 /
    \put{$\bullet$} at 0 0
    \endpicture} at -1.5 -2.598
    %other
    \put{\beginpicture
    \setplotsymbol({\tiny{$\bullet$}})
    \plot 1 0 0.5 0.866 /
    \plot 0.5 0.866 -0.5 0.866 /
    \plot -0.5 0.866 -1 0 /
    \plot -1 0 -0.5 -0.866 /
    \plot -0.5 -0.866 0.5 -0.866 /
    \plot 0.5 -0.866 1 0 /
    \put{$\bullet$} at 0 0
    \endpicture} at 1.5 0.866
    \put{\beginpicture
    \setplotsymbol({\tiny{$\bullet$}})
    \plot 1 0 0.5 0.866 /
    \plot 0.5 0.866 -0.5 0.866 /
    \plot -0.5 0.866 -1 0 /
    \plot -1 0 -0.5 -0.866 /
    \plot -0.5 -0.866 0.5 -0.866 /
    \plot 0.5 -0.866 1 0 /
    \put{$\bullet$} at 0 0
    \endpicture} at 1.5 -0.866
    \put{\beginpicture
    \setplotsymbol({\tiny{$\bullet$}})
    \plot 1 0 0.5 0.866 /
    \plot 0.5 0.866 -0.5 0.866 /
    \plot -0.5 0.866 -1 0 /
    \plot -1 0 -0.5 -0.866 /
    \plot -0.5 -0.866 0.5 -0.866 /
    \plot 0.5 -0.866 1 0 /
    \put{$\bullet$} at 0 0
    \endpicture} at 1.5 2.598
    \put{\beginpicture
    \setplotsymbol({\tiny{$\bullet$}})
    \plot 1 0 0.5 0.866 /
    \plot 0.5 0.866 -0.5 0.866 /
    \plot -0.5 0.866 -1 0 /
    \plot -1 0 -0.5 -0.866 /
    \plot -0.5 -0.866 0.5 -0.866 /
    \plot 0.5 -0.866 1 0 /
    \put{$\bullet$} at 0 0
    \endpicture} at 1.5 -2.598
    %bits
    \put{\beginpicture
    \setplotsymbol({\tiny{$\bullet$}})
    \plot 0.5 -0.866 -0.5 -0.866 /
    \plot -0.5 -0.866 -1 0 /
    \plot -1 0 -0.5 0.866 /
    \plot -0.5 0.866 0.5 0.866 /
    \put{$\bullet$} at 0 0
    \endpicture} at 3 -1.732
    \put{\beginpicture
    \setplotsymbol({\tiny{$\bullet$}})
    \plot 0.5 -0.866 -0.5 -0.866 /
    \plot -0.5 -0.866 -1 0 /
    \plot -1 0 -0.5 0.866 /
    \plot -0.5 0.866 0.5 0.866 /
    \put{$\bullet$} at 0 0
    \endpicture} at 3 0
    \put{\beginpicture
    \setplotsymbol({\tiny{$\bullet$}})
    \plot 0.5 -0.866 -0.5 -0.866 /
    \plot -0.5 0.866 0.5 0.866 /
    \put{$\bullet$} at 0 0
    \endpicture} at 3 1.732
    \put{\beginpicture
    \setplotsymbol({\tiny{$\bullet$}})
    \plot 0.5 -0.866 -0.5 -0.866 /
    \plot -0.5 0.866 0.5 0.866 /
    \put{$\bullet$} at 0 0
    \endpicture} at -3 1.732
    \put{\beginpicture
    \setplotsymbol({\tiny{$\bullet$}})
    \plot 0.5 -0.866 -0.5 -0.866 /
    \plot -0.5 0.866 0.5 0.866 /
    \put{$\bullet$} at 0 0
    \endpicture} at -3 0
    \put{\beginpicture
    \setplotsymbol({\tiny{$\bullet$}})
    \plot 0.5 -0.866 -0.5 -0.866 /
    \plot -0.5 0.866 0.5 0.866 /
    \put{$\bullet$} at 0 0
    \endpicture} at -3 -1.732
    %extending for p-orientation
    %\setdashes
    \plot 2.5 -4.33 3.7 -2.2516 /
    \plot 1.5 -4.33 3.7 -0.52 /
    \plot 0.5 -4.33 3.7 1.212 /
    \plot -0.5 -4.33 3.7 3.044 /
    \plot -1.5 -4.33 3.2 3.81 /
    \plot -2.5 -4.33 2.2 3.81 /
    \plot -3.5 -4.33 1.2 3.81 /
    \plot -3.7 -3.044 0.2 3.81 /
    \plot -3.7 -1.212 -0.8 3.81 /
    \plot -3.7 0.520 -1.8 3.81 /
    \plot -3.7 2.2516 -2.8 3.81 /
    %otherslope
    \plot 2.5 4.33 3.7 2.2516 /
    \plot 1.5 4.33 3.7 0.52 /
    \plot 0.5 4.33 3.7 -1.212 /
    \plot -0.5 4.33 3.7 -3.044 /
    \plot -1.5 4.33 3.2 -3.81 /
    \plot -2.5 4.33 2.2 -3.81 /
    \plot -3.5 4.33 1.2 -3.81 /
    \plot -3.7 3.044 0.2 -3.81 /
    \plot -3.7 1.212 -0.8 -3.81 /
    \plot -3.7 -0.520 -1.8 -3.81 /
    \plot -3.7 -2.2516 -2.8 -3.81 /
    %final slope
    \plot -3.7 -3.464 4.5 -3.464 /
    \plot -3.7 -2.598 4.5 -2.598 /
    \plot -3.7 -1.732 4.5 -1.732 /
    \plot -3.7 -0.866 4.5 -0.866 /
    \plot -3.7 0 4.5 0 /
    \plot -3.7 3.464 4.5 3.464 /
    \plot -3.7 2.598 4.5 2.598 /
    \plot -3.7 1.732 4.5 1.732 /
    \plot -3.7 0.866 4.5 0.866 /
    %the orientation
    %slope 3
    %hyperplane labels
    {\small\put{$H_{-\alpha_1+\delta}$} at -3.5 -4.7
    \put{$H_{\alpha_1}$} at -2.5 -4.7
    %\put{$H_{\alpha_1+\delta}$} at -1.5 -5.2
    \put{$H_{\alpha_1+2\delta}$} at -0.5 -4.7
    %\put{$H_{\alpha_1+3\delta}$} at 0.5 -5.2
    \put{$H_{\alpha_1+4\delta}$} at 1.5 -4.7 } %
    %\put{$H_{\alpha_1+5\delta}$} at 2.5 -5.2 }
    %more
    {\small\put{$H_{\alpha_2+\delta}$} at -3.5 4.7
    \put{$H_{\alpha_2}$} at -2.5 4.7
    %\put{$H_{-\alpha_1+\delta}$} at -1.5 5.1
    \put{$H_{-\alpha_2+2\delta}$} at -0.5 4.7
    %\put{$H_{-\alpha_1+3\delta}$} at 0.5 5.1
    \put{$H_{-\alpha_2+4\delta}$} at 1.5 4.7 } %
    %\put{$H_{-\alpha_1+5\delta}$} at 2.5 5.1 }
    %more
    {\small\put{$H_{-\varphi+4\delta}$}[l] at 4.5 3.464
    \put{$H_{-\varphi+3\delta}$}[l] at 4.5 2.598
    \put{$H_{-\varphi+2\delta}$}[l] at 4.5 1.732
    \put{$H_{\alpha_0}$}[l] at 4.5 0.866
    \put{$H_{\varphi}$}[l] at 4.5 0
    \put{$H_{\varphi+\delta}$}[l] at 4.5 -0.866
    \put{$H_{\varphi+2\delta}$}[l] at 4.5 -1.732
    \put{$H_{\varphi+3\delta}$}[l] at 4.5 -2.598
    \put{$H_{\varphi+4\delta}$}[l] at 4.5 -3.464 }%
    %shading the fundamental alcove
    %\vshade -0.6 0.95 0.85 0.1 0 0.95 /
    %\vshade -0.1 0 0.95 0.6 0.95 0.85 /
        %slope 3
    \put{$+$} at 4.25 0.2
    \put{$+$} at 4.25 1.066
    \put{$+$} at 4.25 1.932
    \put{$+$} at 4.25 2.798
    \put{$+$} at 4.25 3.664
    \put{$-$} at 4.25 -1.066
    \put{$-$} at 4.25 -1.932
    \put{$-$} at 4.25 -2.798
    \put{$-$} at 4.25 -3.664
    \put{$+$} at 4.25 -3.264
    \put{$+$} at 4.25 -2.398
    \put{$+$} at 4.25 -1.532
    \put{$+$} at 4.25 -0.666
    \put{$-$} at 4.25 3.264
    \put{$-$} at 4.25 2.398
    \put{$-$} at 4.25 1.532
    \put{$-$} at 4.25 0.666
    \put{$-$} at 4.25 -0.2
    %slope 1
    \put{$+$} at -0.9 -4.2
    \put{$-$} at -0.4 -4.2
    \put{$+$} at -1.9 -4.2
    \put{$-$} at -1.4 -4.2
    \put{$+$} at -2.9 -4.2
    \put{$-$} at -2.4 -4.2
    \put{$+$} at -3.9 -4.2
    \put{$-$} at -3.4 -4.2
    \put{$+$} at 0.1 -4.2
    \put{$-$} at 0.6 -4.2
    \put{$+$} at 1.1 -4.2
    \put{$-$} at 1.6 -4.2
    \put{$+$} at 2.1 -4.2
    \put{$-$} at 2.6 -4.2
    %slope 3
    \put{$-$} at -0.9 4.2
    \put{$+$} at -0.4 4.2
    \put{$-$} at -1.9 4.2
    \put{$+$} at -1.4 4.2
    \put{$-$} at -2.9 4.2
    \put{$+$} at -2.4 4.2
    \put{$-$} at -3.9 4.2
    \put{$+$} at -3.4 4.2
    \put{$-$} at 0.1 4.2
    \put{$+$} at 0.6 4.2
    \put{$-$} at 1.1 4.2
    \put{$+$} at 1.6 4.2
    \put{$-$} at 2.1 4.2
    \put{$+$} at 2.6 4.2
    %\put{Figure 1: $W_{\mathrm{aff}}$ when $\fg_0=\fsl_3$} at 0 -5.5
\endpicture
\end{equation}
The alcoves are the triangles and the (centres
of) hexagons are the elements of $Q^{\vee}$.

Let $w\in W_{\mathrm{aff}}$.  Following the discussion in 
\eqref{walkseq}-\eqref{rootseq}, a reduced expression 
$\vec w=s_{i_1}\cdots s_{i_{\ell}}$ is
a \emph{walk} starting at $1$ and ending at $w$,  
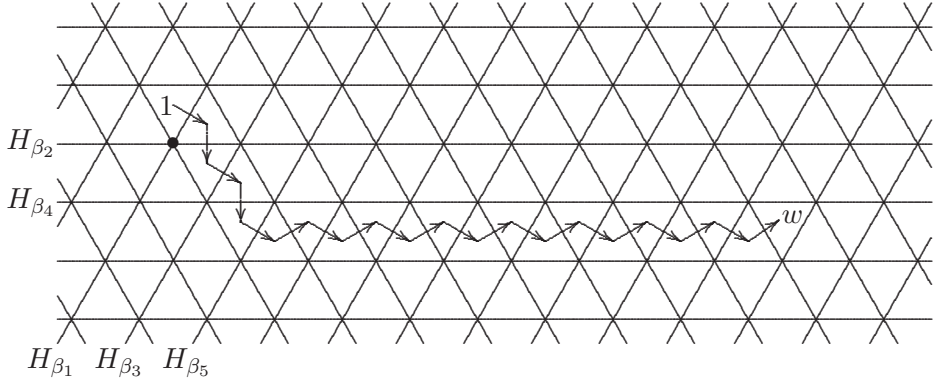
$$\beginpicture
\setcoordinatesystem units <0.9cm,0.9cm>         % sets scale
\setplotarea x from -5 to 10, y from -2 to 4  % sets plot size up
    %slope 1
    %\plot 8.8 -2.078 9.2 -1.3856 /
    \plot 8.3 -1.212 9.2 0.346 /
    \plot 7.3 -1.212 9.2 2.078 /
    \plot 6.3 -1.212 9.2 3.81 /
    \plot 5.3 -1.212 8.2 3.81 /
    \plot 4.3 -1.212 7.2 3.81 /
    \plot 3.3 -1.212 6.2 3.81 /
    \plot 2.3 -1.212 5.2 3.81 /
    \plot 1.3 -1.212 4.2 3.81 /
    \plot 0.3 -1.212 3.2 3.81 /
    \plot -0.7 -1.212 2.2 3.81 /
    \plot -1.7 -1.212 1.2 3.81 /
    \plot -2.7 -1.212 0.2 3.81 /
    \plot -3.7 -1.212 -0.8 3.81 /
    \plot -3.7 0.520 -1.8 3.81 /
    \plot -3.7 2.2516 -2.8 3.81 /
    %slope 2
    %\plot -4.3 -1.212 -3.7 -2.2516 /
    \plot -3.3 -1.212 -3.7 -0.52 /
    \plot -2.3 -1.212 -3.7 1.212 /
    \plot -1.3 -1.212 -3.7 3.044 /
    \plot -0.3 -1.212 -3.2 3.81 /
    \plot 0.7 -1.212 -2.2 3.81 /
    \plot 1.7 -1.212 -1.2 3.81 /
    \plot 2.7 -1.212 -0.2 3.81 /
    \plot 3.7 -1.212 0.8 3.81 /
    \plot 4.7 -1.212 1.8 3.81 /
    \plot 5.7 -1.212 2.8 3.81 /
    \plot 6.7 -1.212 3.8 3.81 /
    \plot 7.7 -1.212 4.8 3.81 /
    \plot 8.7 -1.212 5.8 3.81 /
    \plot 9.2 -0.346 6.8 3.81 /
    \plot 9.2 1.386 7.8 3.81 /
    \plot 9.2 3.118 8.8 3.81 /
    %slope 3 \varphi
    %\plot -3.7 -2.598 9.2 -2.598 /
    %\plot -3.7 -1.732 9.2 -1.732 /
    \plot -3.7 -0.866 9.2 -0.866 /
    \plot -3.7 0 9.2 0 /
    \plot -3.7 3.464 9.2 3.464 /
    \plot -3.7 2.598 9.2 2.598 /
    \plot -3.7 1.732 9.2 1.732 /
    \plot -3.7 0.866 9.2 0.866 /
    %path
    \arrow <5pt> [.2,.67] from -2 2.309 to -1.5 2.021   %
    \arrow <5pt> [.2,.67] from -1.5 2.021 to -1.5 1.443   %
    \arrow <5pt> [.2,.67] from -1.5 1.443 to -1 1.155   %
    \arrow <5pt> [.2,.67] from -1 1.155 to -1 0.577   %
    %fold
    %\put{\beginpicture
    %\setcoordinatesystem units <0.75cm,0.75cm>         % sets scale
    %\plot 0 0 -0.3 -0.15 /
    %\plot -0.3 -0.15 -0.225 -0.3 /
    %\arrow <5pt> [.2,.67] from -0.225 -0.3 to 0.05 -0.13
    %\endpicture} at -1 0.65
    %path
    \arrow <5pt> [.2,.67] from -1 0.577 to -0.5 0.289   %
    %path
    %\put{\beginpicture
    %\setcoordinatesystem units <0.7cm,0.7cm>         % sets scale
    %\arrow <5pt> [.2,.67] from 0.08 -0.32 to 0.08 0 %
    %\plot 0.07 -0.32 -0.08 -0.32 /
    %\plot -0.08 -0.32 -0.08 0 /
    %\endpicture} at -0.5 0.315
    %fold
    \arrow <5pt> [.2,.67] from -0.5 0.289 to 0 0.577   %
    \arrow <5pt> [.2,.67] from 0 0.577 to 0.5 0.289   %
    \arrow <5pt> [.2,.67] from 0.5 0.289 to 1 0.577   %
    \arrow <5pt> [.2,.67] from 1 0.577 to 1.5 0.289   %
    \arrow <5pt> [.2,.67] from 1.5 0.289 to 2 0.577   %
    \arrow <5pt> [.2,.67] from 2 0.577 to 2.5 0.289   %
    \arrow <5pt> [.2,.67] from 2.5 0.289 to 3 0.577   %
    \arrow <5pt> [.2,.67] from 3 0.577 to 3.5 0.289   %
    \arrow <5pt> [.2,.67] from 3.5 0.289 to 4 0.577   %
    \arrow <5pt> [.2,.67] from 4 0.577 to 4.5 0.289   %
    \arrow <5pt> [.2,.67] from 4.5 0.289 to 5 0.577   %
    \arrow <5pt> [.2,.67] from 5 0.577 to 5.5 0.289   %
    \arrow <5pt> [.2,.67] from 5.5 0.289 to 6 0.577   %
    \arrow <5pt> [.2,.67] from 6 0.577 to 6.5 0.289   %
    \arrow <5pt> [.2,.67] from 6.5 0.289 to 6.95 0.6   %
    %hyperplane labels
    %\put{$\bullet$} at -2 1.732
    %\put{$\bullet$} at 7.5 0.866
    %\put{$\scriptstyle{\alpha_1}$} at -3.5 4
    %\put{$\scriptstyle{\alpha_2}$} at -0.5 4
    %\put{$\scriptstyle{\varphi}$}[r] at -3.9 1.732
    %\put{$\scriptstyle{-\alpha_1+9\delta}$}[t] at 8.8 -1.2
    %\put{$\scriptstyle{\alpha_2+9\delta}$}[t] at 5.2 -1.2
    %\put{$\scriptstyle{-\varphi}$}[l] at 9.4 1.732
%    \color{white}
    \put{$\bullet$} at -2 1.732 
%    \color{black}
    \put{$1$} at -2.1 2.3
    \put{$w$} at 7.15 0.65
    \put{$H_{\beta_1}$} at -3.8 -1.5
   \put{$H_{\beta_2}$} at -4.1 1.732
    \put{$H_{\beta_3}$} at -2.8 -1.5
   \put{$H_{\beta_4}$} at -4.1 0.866
    \put{$H_{\beta_5}$} at -1.8 -1.5
\endpicture$$
and the points of 
\begin{align}\label{eq:points}
IwI = \{ x_{i_1}(c_1)n_{i_1}^{-1}x_{i_2}(c_2)n_{i_2}^{-1}\cdots
x_{i_\ell}(c_\ell)n_{i_\ell}^{-1}I\ |\ c_1,\ldots, c_\ell\in \CC\}
\end{align}
are in bijection with labelings of the 
edges of the walk by complex numbers $c_1,\ldots, c_\ell$.
The elements of $R(w)=\{\beta_1,\ldots, \beta_\ell\}$ are 
the elements of $\tilde R_{\mathrm{re}}^I$ corresponding to the sequence of hyperplanes 
crossed by the walk.

The labeling of the hyperplanes in \eqref{sl3pict1} is such that neighboring alcoves have
\begin{equation}\label{hyplabels}
\beginpicture
\setcoordinatesystem units <0.8cm,0.8cm>         % sets scale
\setplotarea x from -0.8 to 0.7, y from -0.5 to 0.5  % sets plot size up
\put{$H_{v\alpha_j}$}[b] at 0 0.6
%\put{$\scriptstyle{-}$}[b] at -0.4 0.25
%\put{$\scriptstyle{+}$}[b] at 0.4 0.25
%\put{$\scriptstyle{c}$}[tl] at 0.1 -0.1
\put{$\scriptstyle{v}$}[br] at -0.6 0.1
\put{$\scriptstyle{vs_j}$}[bl] at 0.6 0.1
\plot  0 -0.4  0 0.5 /
\arrow <5pt> [.2,.67] from -0.5 0 to 0.5 0   %
\endpicture
\qquad\hbox{with $v\alpha_j\in \tilde R_{\mathrm{re}}^I$ if $v$ is closer to $1$ than $vs_j$.}
\end{equation}
The \emph{periodic orientation} (illustrated in (\ref{sl3pict1})) is the orientation of the hyperplanes $H_{\alpha+k\delta}$
such that 
\begin{enumerate}
\item[(a)] $1$ is on the positive side of $H_\alpha$ for $\alpha\in R_{\mathrm{re}}^+$,
\item[(b)] $H_{\alpha+k\delta}$ and $H_\alpha$ have parallel orientiations.
\end{enumerate}
This orientation is such that 
\begin{equation}\label{orttn}
v\alpha_j\in \tilde R_{\mathrm{re}}^U
\quad\hbox{if and only if}\quad
\beginpicture
\setcoordinatesystem units <0.8cm,0.8cm>         % sets scale
\setplotarea x from -0.8 to 0.7, y from -0.5 to 0.5  % sets plot size up
\put{$H_{v\alpha_j}$}[b] at 0 0.6
\put{$\scriptstyle{-}$}[b] at -0.4 0.25
\put{$\scriptstyle{+}$}[b] at 0.4 0.25
%\put{$\scriptstyle{c}$}[tl] at 0.1 -0.1
\put{$\scriptstyle{v}$}[br] at -0.6 0.1
\put{$\scriptstyle{vs_j}$}[bl] at 0.6 0.1
\plot  0 -0.4  0 0.5 /
\arrow <5pt> [.2,.67] from -0.5 0 to 0.5 0   %
\endpicture
.
\end{equation}

Together, \eqref{hyplabels} and \eqref{orttn} provide a powerful combinatorics for 
analyzing the intersections $U^-vI\cap IwI$.
We shall use the first identity in \eqref{nalphadefn}, in the form
\begin{equation}\label{foldinglaw}
x_\alpha(c)n_\alpha^{-1} = x_{-\alpha}(c^{-1})x_{\alpha}(-c)h_{\alpha^\vee}(c)
\qquad\hbox{(main folding law)},
\end{equation}
to rewrite the points of $IwI$ given in (\ref{eq:points}) as elements of $U^-vI$.
Suppose that
\begin{equation}\label{startstate}
x_{i_1}(c_1)n_{i_1}^{-1}\cdots x_{i_\ell}(c_\ell)n_{i_\ell}^{-1}
= x_{\gamma_1}(c_1')\cdots x_{\gamma_\ell}(c_\ell')n_vb,
\qquad\hbox{where $b\in I$},
\end{equation}
$v\in W_{\mathrm{aff}}$ and $n_v = n_{j_1}^{-1}\cdots n_{j_k}^{-1}$ 
if $v=s_{i_1}\cdots s_{i_k}$
is a reduced word,
and $\gamma_1,\ldots, \gamma_\ell\in \tilde R_{\mathrm{re}}^U$ so that 
$x_{\gamma_1}(c_1')\cdots x_{\gamma_\ell}(c_\ell')\in U^-$. Then the procedure described in \eqref{type1step}-\eqref{type3step} will compute
$c_{\ell+1}'\in \CC$, $b'\in I$, $v'\in W_{\mathrm{aff}}$ and 
$\gamma_{\ell+1}\in \tilde R_{\mathrm{re}}^U$ so that
$$x_{i_1}(c_1)n_{i_1}^{-1}\cdots x_{i_\ell}(c_\ell)n_{i_\ell}^{-1}x_j(c)n_j^{-1}
= x_{\gamma_1}(c_1')\cdots x_{\gamma_\ell}(c_\ell')
x_{\gamma_{\ell+1}}(c_{\ell+1})n_{v'}b'.
$$

Keep the notations in \eqref{startstate}.
Since $bx_j(c)n_j^{-1}\in Is_jI$  there are unique $\tilde c\in \CC$ and $b'\in I$ such that 
$bx_j(c)n_j^{-1}=x_j(\tilde c)n_j^{-1}b'$ and
\begin{align*}
x_{i_1}(c_1)n_{i_1}^{-1}\cdots x_{i_\ell}(c_\ell)n_{i_\ell}^{-1}x_j(c)n_j^{-1}
&= x_{\gamma_1}(c_1')\cdots x_{\gamma_\ell}(c_\ell')n_vbx_j(c)n_j^{-1} \\
&= x_{\gamma_1}(c_1')\cdots x_{\gamma_\ell}(c_\ell')n_vx_j(\tilde c)n_j^{-1}b'.
\end{align*}
\emph{Case 1}: If $v\alpha_j\in \tilde R_{\mathrm{re}}^U$, \quad
$\beginpicture
\setcoordinatesystem units <0.8cm,0.8cm>         % sets scale
\setplotarea x from -0.8 to 0.7, y from -0.5 to 0.5  % sets plot size up
\put{$H_{v\alpha_j}$}[b] at 0 0.6
\put{$\scriptstyle{-}$}[b] at -0.4 0.25
\put{$\scriptstyle{+}$}[b] at 0.4 0.25
\put{$\scriptstyle{\tilde c}$}[tl] at 0.1 -0.1
\put{$\scriptstyle{v}$}[br] at -0.6 0.1
\put{$\scriptstyle{vs_j}$}[bl] at 0.6 0.1
\plot  0 -0.4  0 0.5 /
\arrow <5pt> [.2,.67] from -0.5 0 to 0.5 0   %
\endpicture
$,
\qquad then $x_{\gamma_1}(c_1')\cdots x_{\gamma_\ell}(c_\ell')n_vx_j(\tilde c)n_j^{-1}b'$ is equal to 
$$
x_{\gamma_1}(c_1')\cdots x_{\gamma_\ell}(c_\ell')x_{v\alpha_j}(\pm\tilde c)n_{vs_j}b'
\in U^-vs_jI\cap Iws_jI.
$$
In this case, $\gamma_{\ell+1}=v\alpha_j$, $v'=vs_j$, and
\begin{equation}\label{type1step}
\beginpicture
\setcoordinatesystem units <0.8cm,0.8cm>         % sets scale
\setplotarea x from -0.8 to 0.7, y from -0.5 to 0.5  % sets plot size up
\put{$H_{v\alpha_j}$}[b] at 0 0.6
\put{$\scriptstyle{-}$}[b] at -0.4 0.25
\put{$\scriptstyle{+}$}[b] at 0.4 0.25
\put{$\scriptstyle{\tilde c}$}[tl] at 0.1 -0.1
\put{$\scriptstyle{v}$}[br] at -0.6 0.1
\put{$\scriptstyle{vs_j}$}[bl] at 0.6 0.1
\plot  0 -0.4  0 0.5 /
\arrow <5pt> [.2,.67] from -0.5 0 to 0.5 0   %
\endpicture
\qquad\hbox{becomes}\qquad
\beginpicture
\setcoordinatesystem units <0.8cm,0.8cm>         % sets scale
\setplotarea x from -0.8 to 0.7, y from -0.5 to 0.5  % sets plot size up
\put{$H_{v\alpha_j}$}[b] at 0 0.6
\put{$\scriptstyle{-}$}[b] at -0.4 0.25
\put{$\scriptstyle{+}$}[b] at 0.4 0.25
\put{$\scriptstyle{\pm\tilde c}$}[tl] at 0.1 -0.1
\put{$\scriptstyle{v}$}[br] at -0.6 0.1
\put{$\scriptstyle{vs_j}$}[bl] at 0.6 0.1
\plot  0 -0.4  0 0.5 /
\arrow <5pt> [.2,.67] from -0.5 0 to 0.5 0   %
\endpicture
.
\end{equation}
\emph{Case 2:} If $v\alpha_j\not\in \tilde{R}_{\mathrm{re}}^U$ and $\tilde c\ne 0$, \quad
$\beginpicture
\setcoordinatesystem units <0.8cm,0.8cm>         % sets scale
\setplotarea x from -0.8 to 0.7, y from -0.5 to 0.5  % sets plot size up
\put{$H_{v\alpha_j}$}[b] at 0 0.6
\put{$\scriptstyle{-}$}[b] at -0.4 0.25
\put{$\scriptstyle{+}$}[b] at 0.4 0.25
\put{$\scriptstyle{\tilde c}$}[tl] at 0.1 -0.1
\put{$\scriptstyle{v}$}[bl] at 0.6 0.1
\put{$\scriptstyle{vs_j}$}[br] at -0.6 0.1
\plot  0 -0.4  0 0.5 /
\arrow <5pt> [.2,.67] from 0.5 0 to -0.5 0   %
\endpicture
$,
\qquad then
\begin{align*}
x_{\gamma_1}(c_1')\cdots x_{\gamma_\ell}(c_\ell')n_vx_{\alpha_j}(\tilde c)n_j^{-1}b'
&=
x_{\gamma_1}(c_1')\cdots x_{\gamma_\ell}(c_\ell')n_v
x_{-\alpha_j}(\tilde c^{-1})x_{\alpha_j}(-\tilde c)h_{\alpha_j^\vee}(\tilde c)b' \\
&=
x_{\gamma_1}(c_1')\cdots x_{\gamma_\ell}(c_\ell')n_vx_{-\alpha_j}(\tilde c^{-1})b'' \\
&=
x_{\gamma_1}(c_1')\cdots x_{\gamma_\ell}(c_\ell')x_{\gamma_{\ell+1}}(\pm\tilde c^{-1})n_vb''
\in U^-vI\cap Iws_jI,
\end{align*}
where $\gamma_{\ell+1} = -v\alpha_j$ and $b'' = x_{\alpha_j}(-\tilde{c})h_{\alpha_j^{\vee}}(\tilde{c})b'$.  So
\begin{equation}\label{type2step}
\beginpicture
\setcoordinatesystem units <0.8cm,0.8cm>         % sets scale
\setplotarea x from -0.8 to 0.7, y from -0.5 to 0.5  % sets plot size up
\put{$H_{v\alpha_j}$}[b] at 0 0.6
\put{$\scriptstyle{-}$}[b] at -0.4 0.25
\put{$\scriptstyle{+}$}[b] at 0.4 0.25
\put{$\scriptstyle{\tilde{c}}$}[tl] at 0.1 -0.1
\put{$\scriptstyle{v}$}[bl] at 0.6 0.1
\put{$\scriptstyle{vs_j}$}[br] at -0.6 0.1
\plot  0 -0.4  0 0.5 /
\arrow <5pt> [.2,.67] from 0.5 0 to -0.5 0   %
\endpicture
\qquad\hbox{becomes}\qquad
\beginpicture
\setcoordinatesystem units <0.8cm,0.8cm>         % sets scale
\setplotarea x from -0.8 to 0.7, y from -0.5 to 0.5  % sets plot size up
\put{$H_{v\alpha_j}$}[b] at 0 0.7
\put{$\scriptstyle{-}$}[b] at -0.4 0.35
\put{$\scriptstyle{+}$}[b] at 0.4 0.35
\put{$\scriptstyle{\pm\tilde{c}^{-1}}$}[tl] at 0.1 -0.1
\put{$\scriptstyle{v}$}[bl] at 0.6 0.1
\plot  0 -0.4  0 0.6 /
\plot 0.5 0  0.05 0 /
\arrow <5pt> [.2,.67] from 0.05 0.1 to 0.5 0.1   %
\plot 0.05 0 0.05 0.1 /
\endpicture
\end{equation}

\smallskip\noindent
\emph{Case 3:}  If  $v\alpha_j\not\in \tilde R_{\mathrm{re}}^U$ and $\tilde c=0$, \quad
$\beginpicture
\setcoordinatesystem units <0.8cm,0.8cm>         % sets scale
\setplotarea x from -0.8 to 0.7, y from -0.5 to 0.5  % sets plot size up
\put{$H_{v\alpha_j}$}[b] at 0 0.6
\put{$\scriptstyle{-}$}[b] at -0.4 0.25
\put{$\scriptstyle{+}$}[b] at 0.4 0.25
\put{$\scriptstyle{0}$}[tl] at 0.1 -0.1
\put{$\scriptstyle{v}$}[bl] at 0.6 0.1
\put{$\scriptstyle{vs_j}$}[br] at -0.6 0.1
\plot  0 -0.4  0 0.5 /
\arrow <5pt> [.2,.67] from 0.5 0 to -0.5 0   %
\endpicture
$,
\qquad then
\begin{align*}
x_{\gamma_1}(c_1')\cdots x_{\gamma_\ell}(c_\ell')n_vx_{\alpha_j}(0)n_j^{-1}b'
&=x_{\gamma_1}(c_1')\cdots x_{\gamma_\ell}(c_\ell')n_vx_{-\alpha_j}(0)n_j^{-1}b' \\
&=x_{\gamma_1}(c_1')\cdots x_{\gamma_\ell}(c_\ell')x_{\gamma_{\ell+1}}(0)n_{vs_j}b'
\in U^-vs_jI\cap Iws_jI,
\end{align*}
where $\gamma_{\ell+1} = -v\alpha_j$.  So
\begin{equation}\label{type3step}
\beginpicture
\setcoordinatesystem units <0.8cm,0.8cm>         % sets scale
\setplotarea x from -0.8 to 0.7, y from -0.5 to 0.5  % sets plot size up
\put{$H_{v\alpha_j}$}[b] at 0 0.6
\put{$\scriptstyle{-}$}[b] at -0.4 0.25
\put{$\scriptstyle{+}$}[b] at 0.4 0.25
\put{$\scriptstyle{0}$}[tl] at 0.1 -0.1
\put{$\scriptstyle{v}$}[bl] at 0.6 0.1
\put{$\scriptstyle{vs_j}$}[br] at -0.6 0.1
\plot  0 -0.4  0 0.5 /
\arrow <5pt> [.2,.67] from 0.5 0 to -0.5 0   %
\endpicture
\qquad\hbox{becomes}\qquad
\beginpicture
\setcoordinatesystem units <0.8cm,0.8cm>         % sets scale
\setplotarea x from -0.8 to 0.7, y from -0.5 to 0.5  % sets plot size up
\put{$H_{v\alpha_j}$}[b] at 0 0.6
\put{$\scriptstyle{-}$}[b] at -0.4 0.25
\put{$\scriptstyle{+}$}[b] at 0.4 0.25
\put{$\scriptstyle{0}$}[tl] at 0.1 -0.1
\put{$\scriptstyle{v}$}[bl] at 0.6 0.1
\put{$\scriptstyle{vs_j}$}[br] at -0.6 0.1
\plot  0 -0.4  0 0.5 /
\arrow <5pt> [.2,.67] from 0.5 0 to -0.5 0   %
\endpicture
\end{equation}
We have proved the following theorem.

\begin{thm}\label{mainthm}
If $w\in W_{\mathrm{aff}}$ and
$\vec w=s_{i_1}\cdots s_{i_{\ell}}$ is a minimal length walk to $w$ define
$$\cP(\vec w)_{v}=\left\{\begin{matrix}\textrm{labeled folded paths $p$ of type $\vec w$}\\ 
\textrm{which end in $v$}\end{matrix}\right\}
\qquad\hbox{for $v\in W_{\mathrm{aff}}$,}
$$
where a \emph{labeled folded path of type $\vec w$} is a sequence of steps 
of the form 
$$
\beginpicture
\setcoordinatesystem units <0.8cm,0.8cm>         % sets scale
\setplotarea x from -0.8 to 0.7, y from -0.5 to 0.5  % sets plot size up
\put{$H_{v\alpha_j}$}[b] at 0 0.6
\put{$\scriptstyle{-}$}[b] at -0.4 0.25
\put{$\scriptstyle{+}$}[b] at 0.4 0.25
\put{$\scriptstyle{c}$}[tl] at 0.1 -0.1
\put{$\scriptstyle{v}$}[br] at -0.6 0.1
\put{$\scriptstyle{vs_j}$}[bl] at 0.6 0.1
\plot  0 -0.4  0 0.5 /
\arrow <5pt> [.2,.67] from -0.5 0 to 0.5 0   %
\endpicture
,
\qquad
\beginpicture
\setcoordinatesystem units <0.8cm,0.8cm>         % sets scale
\setplotarea x from -0.8 to 0.7, y from -0.5 to 0.5  % sets plot size up
\put{$H_{v\alpha_j}$}[b] at 0 0.7
\put{$\scriptstyle{-}$}[b] at -0.4 0.35
\put{$\scriptstyle{+}$}[b] at 0.4 0.35
\put{$\scriptstyle{c^{-1}}$}[tl] at 0.1 -0.1
\put{$\scriptstyle{v}$}[bl] at 0.6 0.1
\plot  0 -0.4  0 0.6 /
\plot 0.5 0  0.05 0 /
\arrow <5pt> [.2,.67] from 0.05 0.1 to 0.5 0.1   %
\plot 0.05 0 0.05 0.1 /
\endpicture
,
\qquad
\beginpicture
\setcoordinatesystem units <0.8cm,0.8cm>         % sets scale
\setplotarea x from -0.8 to 0.7, y from -0.5 to 0.5  % sets plot size up
\put{$H_{v\alpha_j}$}[b] at 0 0.6
\put{$\scriptstyle{-}$}[b] at -0.4 0.25
\put{$\scriptstyle{+}$}[b] at 0.4 0.25
\put{$\scriptstyle{0}$}[tl] at 0.1 -0.1
\put{$\scriptstyle{v}$}[bl] at 0.6 0.1
\put{$\scriptstyle{vs_j}$}[br] at -0.6 0.1
\plot  0 -0.4  0 0.5 /
\arrow <5pt> [.2,.67] from 0.5 0 to -0.5 0   %
\endpicture
,
\qquad
\hbox{where the $k$th step has $j=i_k$.}
$$
Viewing $U^-vI\cap IwI$ as a subset of $G/I$, 
there is a bijection
$$\cP(\vec w)_v\longleftrightarrow U^-vI\cap IwI.
%\qquad\hbox{and}\qquad G = \bigsqcup_{v\in\widetilde{W}} U^-vI.
$$
\end{thm}

\section{An example}  

For the group $G = SL_3(\CC((t)))$,
$$x_{\alpha_1}(c) = \begin{pmatrix} 1 &c &0 \\ 0 &1 &0 \\ 0 &0 &1\end{pmatrix}
\quad
h_{\alpha_1^\vee}(c) = \begin{pmatrix} c &0 &0 \\ 0 &c^{-1} &0 \\ 0 &0 &1\end{pmatrix},
\quad
n_1 = \begin{pmatrix} 0 &1 &0 \\ -1 &0 &0 \\ 0 &0 &1\end{pmatrix},
%\quad
%n_1^{-1} = \begin{pmatrix} 0 &-1 &0 \\ 1 &0 &0 \\ 0 &0 &1\end{pmatrix}
$$
%and the folding law is
%$$x_1(c)n_1^{-1} = x_{-\alpha_1}(c^{-1})x_{\alpha_1}(-c)h_{\alpha_1^\vee}(c)
%= x_{-\alpha_1}(c^{-1})\begin{pmatrix} c &-1 &0 \\ 0 &c^{-1} &0 \\ 0 &0 &1\end{pmatrix}.
%$$
$$x_{\alpha_2}(c) = \begin{pmatrix} 1 &0 &0 \\ 0 &1 &c \\ 0 &0 &1\end{pmatrix}
\quad
h_{\alpha_2^\vee}(c) = \begin{pmatrix} 1 &0 &0 \\ 0 &c &0 \\ 0 &0 &c^{-1}\end{pmatrix},
\quad
n_2 = \begin{pmatrix} 1 &0 &0 \\ 0 &0 &1 \\ 0 &-1 &0\end{pmatrix},
%\quad
%n_2^{-1} = \begin{pmatrix} 1 &0 &0 \\ 0 &0 &1 \\ 0 &-1 &0\end{pmatrix}
$$
%and
$$x_{\alpha_0}(c) = \begin{pmatrix} 1 &0 &0 \\ 0 &1 &0 \\ ct &0 &1\end{pmatrix}
\quad
h_{\alpha_0^\vee}(c) = \begin{pmatrix} c^{-1} &0 &0 \\ 0 &1 &0 \\ 0 &0 &c\end{pmatrix},
\quad
n_0 = \begin{pmatrix} 0 &0 &-t^{-1} \\ 0 &1 &0 \\ t &0 &0\end{pmatrix}.
%\quad
%n_0^{-1} = \begin{pmatrix} 0 &0 &t^{-1} \\ 0 &1 &0 \\ -t &0 &0\end{pmatrix}
$$
%and the folding law is
%$$x_0(c)n_0^{-1} = x_{-\alpha_0}(c^{-1})x_{\alpha_0}(-c)h_{\alpha_0^\vee}(c)
%= x_{-\alpha_0}(c^{-1})\begin{pmatrix} c^{-1} &0 &0 \\ 0 &1 &0 \\ -t &0 &c\end{pmatrix}.
%$$

Let 
$w=s_2s_1s_0s_2s_0s_1s_0s_2s_0$ and $v=s_2s_1s_0s_2s_1s_2s_0$ so that
$$w = \begin{pmatrix} t^2 &0 &0 \\ 0 &0 &1 \\ 0 &t^{-2} &0 \end{pmatrix}
\quad\hbox{and}\quad
v = \begin{pmatrix} 0 &1 &0 \\ t^2 &0 &0 \\ 0 &0 &t^{-2}\end{pmatrix}.
$$
We shall use Theorem \ref{mainthm} to show that the points of $IwI\cap U^-vI$ are
$$x_2(c_1)n_2^{-1}x_1(c_2)n_1^{-1}x_0(c_3)n_0^{-1}
x_2(c_4)n_2^{-1}x_0(c_5)n_0^{-1}x_1(c_6)n_1^{-1}
x_0(c_7)n_0^{-1}x_2(c_8)n_2^{-1}x_0(c_9)n_0^{-1}I,
$$
with $c_1,\ldots, c_9\in \CC$ such that
\begin{equation}\label{oldlabels}
c_1=0,\ \  c_2=0,\ \  c_3=0,\ \ c_4=0,\ \  c_5\ne 0,\ \  c_6=0,\ \  c_7\ne 0,\ \  c_9=c_7^{-1}c_8.
\end{equation}
Precisely,
$$x_2(0)n_2^{-1}x_1(0)n_1^{-1}x_0(0)n_0^{-1}
x_2(0)n_2^{-1}x_0(c_5)n_0^{-1}x_1(0)n_1^{-1}
x_0(c_7)n_0^{-1}x_2(c_8)n_2^{-1}x_0(c_7^{-1}c_8)n_0^{-1}$$
is equal to $u_9v_9b_9$,
with $u_9\in U^-$, $v_9\in N$, $b_9\in I$ given by
%$u_9 = x_{-\alpha_2}(0)x_{-\varphi}(0)x_{-\alpha_2-\delta}(0)x_{-\varphi-\delta}(0)
%x_{-\alpha_1}(c_5^{-1})x_{-\alpha_2-2\delta}(0)
%x_{-\varphi-2\delta}(c_5^{-1}c_7^{-1})
%x_{-\alpha_1+\delta}(-c_5^{-2}c_7^{-1}c_8)
%x_{-\alpha_2-3\delta}(0)$,
\begin{equation}\label{eq:finalstats}
\begin{array}{c}
u_9=\begin{pmatrix}
1 &0 &0 \\ c_5^{-1}-c_5^{-2}c_7^{-1}c_8t &1 &0 \\
c_5^{-1}c_7^{-1}t^{-2} &0 &1\end{pmatrix}, 
\qquad
v_9=\begin{pmatrix}0&1&0\\
-t^2&0&0\\
0&0&t^{-2}\end{pmatrix} \\
b_9=\begin{pmatrix}
c_5^{-1}-c_5^{-2}c_7^{-1}c_8t&-c_5^{-2}c_7^{-1}c_8^2&c_5^{-2}c_7^{-2}c_8^2 \\
-t^2&c_5c_7+c_8t&-c_5-c_7^{-1}c_8t \\
-c_5^{-1}c_7^{-1}t^2&-c_5^{-1}c_7^{-1}c_8t&c_7^{-1}+c_5^{-1}c_7^{-2}c_8t\end{pmatrix},
\end{array}
\end{equation}
so that
$u_9=x_{-\alpha_2}(d_1)x_{-\varphi}(d_2)x_{-\alpha_2-\delta}(d_3)x_{-\varphi-\delta}(d_4)
x_{-\alpha_1}(d_5)x_{-\alpha_2-2\delta}(d_6)x_{-\varphi-3\delta}(d_7)
x_{-\alpha_1+\delta}(d_8)$ $\cdot x_{-\alpha_2-3\delta}(d_9)$ 
with
\begin{equation}\label{newlabels}
d_1=d_2=d_3=d_4=0,\ \ d_5=c_5^{-1},\ \ d_6=0,\ \ d_7=c_5^{-1}c_7^{-1},\ \ 
d_8=-c_5^{-2}c_7^{-1}c_8,\ \ d_9=0.
\end{equation}
Pictorially, the walk
with labels $c_1,\ldots, c_9$
$$
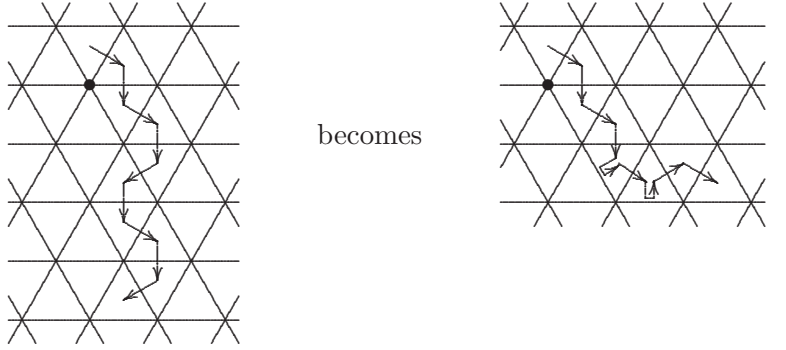
\beginpicture
\setcoordinatesystem units <0.9cm,0.9cm>         % sets scale
\setplotarea x from -2 to 2, y from -3.3 to 2.35  % sets plot size up
    %slope 1 \alpha_2
    \plot 1.3 -2.944 1.7 -2.2516 /
    \plot 0.3 -2.944 1.7 -0.5196 /
    \plot -0.7 -2.944 1.7 1.2124 /
    \plot -1.7 -2.944 1.2 2.0784 /
    \plot -1.7 -1.2124 0.2 2.0784 /
    \plot -1.7 0.5196 -0.8 2.0784 /
    \plot -1.3 -2.944 -1.7 -2.2516 /
    \plot -0.3 -2.944 -1.7 -0.5196 /
    \plot 0.7 -2.944 -1.7 1.2124 /
    \plot 1.7 -2.944 -1.2 2.0784 /
    \plot 1.7 -1.2124 -0.2 2.0784 /
    \plot 1.7 0.5196 0.8 2.0784 /
    \plot -1.7 -2.598 1.7 -2.598 /
    \plot -1.7 -1.732 1.7 -1.732 /
    \plot -1.7 -0.866 1.7 -0.866 /
    \plot -1.7 0 1.7 0 /
    \plot -1.7 1.732 1.7 1.732 /
    \plot -1.7 0.866 1.7 0.866 /
    %path
    \arrow <5pt> [.2,.67] from -0.5 1.443 to 0 1.155   %
    \arrow <5pt> [.2,.67] from 0 1.155 to 0 0.577   %
    \arrow <5pt> [.2,.67] from 0 0.577 to 0.5 0.289   %
    \arrow <5pt> [.2,.67] from 0.5 0.289 to 0.5 -0.289   %
    \arrow <5pt> [.2,.67] from 0.5 -0.289 to 0 -0.577   %
    \arrow <5pt> [.2,.67] from 0 -0.577 to 0 -1.155   %
    \arrow <5pt> [.2,.67] from 0 -1.155 to 0.5 -1.443   %
    \arrow <5pt> [.2,.67] from 0.5 -1.443 to 0.5 -2.021   %
    \arrow <5pt> [.2,.67] from 0.5 -2.021 to 0 -2.309   %
%    \arrow <5pt> [.2,.67] from 0.5 -1.443 to 0.5 -2.021   %
    \put{$\bullet$} at -0.5 0.866
\endpicture
\qquad\hbox{becomes}\qquad
\beginpicture
\setcoordinatesystem units <0.9cm,0.9cm>         % sets scale
\setplotarea x from -2.5 to 2, y from -1.2 to 2.1  % sets plot size up
    %slope 1
    \plot -2.2 1.3856 -1.8 2.0784 /
    \plot -2.2 -0.3464 -0.8 2.0784 /
    \plot -1.7 -1.2124 0.2 2.0784 /
    \plot -0.7 -1.2124 1.2 2.0784 /
    \plot 0.3 -1.2124 1.7 1.2124 /
    \plot 1.3 -1.2124 1.7 -0.5196 /
    %slope 2
    \plot 0.8 2.0784 1.7 0.5196 /
    \plot -0.2 2.0784 1.7 -1.2124 /
    \plot -1.2 2.0784 0.7 -1.2124 /
    \plot -2.2 2.0784 -0.3 -1.2124 /
    \plot -2.2 0.3464 -1.3 -1.2124 /
    %slope 3 \varphi
    \plot -2.2 -0.866 1.7 -0.866 /
    \plot -2.2 0 1.7 0 /
    \plot -2.2 1.732 1.7 1.732 /
    \plot -2.2 0.866 1.7 0.866 /
    %path
    \arrow <5pt> [.2,.67] from -1.5 1.443 to -1 1.155   %
    \arrow <5pt> [.2,.67] from -1 1.155 to -1 0.577   %
    \arrow <5pt> [.2,.67] from -1 0.577 to -0.5 0.289   %
    \arrow <5pt> [.2,.67] from -0.5 0.289 to -0.5 -0.216   %
    %fold
    \put{\beginpicture
    \setcoordinatesystem units <0.75cm,0.75cm>         % sets scale
    \plot 0 0 -0.28 -0.15 /
    \plot -0.28 -0.15 -0.2 -0.3 /
    \arrow <5pt> [.2,.67] from -0.2 -0.3 to 0.05 -0.13
    \endpicture} at -0.5 -0.216
    %path
    \arrow <5pt> [.2,.67] from -0.45 -0.289 to -0.05 -0.566   %
    %path
    \put{\beginpicture
    \setcoordinatesystem units <0.7cm,0.7cm>         % sets scale
    \arrow <5pt> [.2,.67] from 0.08 -0.32 to 0.08 0 %
    \plot 0.07 -0.32 -0.08 -0.32 /
    \plot -0.08 -0.32 -0.08 0 /
    \endpicture} at 0 -0.551
    %fold
    \arrow <5pt> [.2,.67] from 0.05 -0.566 to 0.5 -0.289   %   
    \arrow <5pt> [.2,.67] from 0.5 -0.289 to 1 -0.577   %
    \put{$\bullet$} at -1.5 0.866
\endpicture,$$
the labeled folded path with labels $d_1,\ldots,d_9$.

The step by step computation is as follows:

\smallskip\noindent
Step 1:  If $c_1=0$ then
$$x_2(c_1)n_2^{-1} = x_{-\alpha_2}(0)n_2^{-1}=u_1v_1b_1,\qquad\hbox{with}$$
$$u_1= x_{-\alpha_2}(0),\qquad v_1 = \begin{pmatrix}
1 &0 &0 \\ 0 &0 &-1 \\ 0 &1 &0 \end{pmatrix},
\qquad\hbox{and}\qquad
b_1=1.
$$
Step 2: If $c_2=0$ then, since $v_1x_1(c_2)v_1^{-1} = x_\varphi(c_2)$,
$$
u_1v_1b_1x_1(c_2)n_1^{-1} =u_1x_{\varphi}(c_2)v_1n_1^{-1}b_1= u_1x_{-\varphi}(0)v_1n_1^{-1}b_1 = u_2v_2b_2, \qquad\hbox{with}$$
$$u_2 = u_1x_{-\varphi}(0),\qquad 
v_2 = v_1n_1^{-1} = \begin{pmatrix} 0 &-1 &0 \\ 0 &0 &-1 \\ 1 &0 &0\end{pmatrix}
\qquad\hbox{and}\qquad
b_2=1.$$

\smallskip\noindent
Step 3:  If $c_3 = 0$ then, since $v_2x_0(c_3)v_2^{-1} = x_{\alpha_2+\delta}(-c_3)$,
$$u_2v_2b_2x_0(c_3)n_0^{-1} =u_2x_{\alpha_2+\delta}(-c_3)v_2n_0^{-1}b_2= u_2x_{-\alpha_2-\delta}(0)v_2n_0^{-1} b_2= u_3 v_3b_3, \qquad\hbox{with}
$$
$$u_3 = u_2x_{-\alpha_2-\delta}(0),\qquad
v_3 = v_2n_0^{-1} = \begin{pmatrix} 0 &-1 &0 \\ t &0 &0 \\ 0 &0 &t^{-1}\end{pmatrix},
\qquad\hbox{and}\qquad
b_3=1.
$$

\smallskip\noindent
Step 4: If $c_4=0$ then, since $v_3x_2(c_4)v_3^{-1} = x_{\varphi+\delta}(-c_4)$,
$$u_3v_3b_3x_2(c_4)n_2^{-1} = u_3x_{\varphi+\delta}(-c_4)v_3n_2^{-1}b_3=u_3x_{-\varphi-\delta}(0)v_3n_2^{-1}b_3 = u_4v_4b_4,
\qquad\hbox{with}
$$
$$u_4 = u_3x_{-\varphi-\delta}(0),\qquad
v_4 = v_3n_2^{-1} = \begin{pmatrix} 0 &0 &1\\ t &0 &0 \\ 0 &t^{-1} &0 \end{pmatrix}
\qquad\hbox{and}\qquad
b_4=1.$$

\smallskip\noindent
Step 5:  If $c_5\ne 0$ then by the folding law and the fact that $v_4x_{-\alpha_0}(c_5^{-1})v_4^{-1} = x_{-\alpha_1}(c_5^{-1})$,
$$u_4v_4b_4x_0(c_5)n_0^{-1} = u_4v_4x_{-\alpha_0}(c_5^{-1}) x_{\alpha_0}(-c_5)
h_{\alpha_0^\vee}(c_5)b_4
=u_4x_{-\alpha_1}(c_5^{-1})v_4b_5 = u_5v_5b_5,
$$
where
$$u_5=u_4x_{-\alpha_1}(c_5^{-1}),\qquad v_5=v_4,\qquad\textrm{and}\qquad b_5 = x_{\alpha_0}(-c_5)h_{\alpha_0^\vee}(c_5)b_4
=\begin{pmatrix} c_5^{-1} &0 &0 \\ 0 &1 &0 \\ -t &0 &c_5\end{pmatrix}.
$$

\smallskip\noindent
Step 6:  If $c_5^{-1}c_6= 0$ (so $c_6=0$) then
$$u_5v_5b_5x_1(c_6)n_1^{-1}= u_5v_5x_1(c_5^{-1}c_6)n_1^{-1}b_5'
=u_5x_{-\alpha_2-2\delta}(0)v_5n_1^{-1}b_5' = u_6v_6b_6,
$$
with 
$$u_6=u_5x_{-\alpha_2-2\delta}(0),\quad v_6 = v_5n_1^{-1}=\begin{pmatrix} 0 &0 &1 \\ 0 &-t &0 \\ t^{-1} &0 &0\end{pmatrix}
\quad\hbox{and}\quad
b_6=b_5' = \begin{pmatrix} 1 &0 &0 \\ 0 &c_5^{-1} &0 \\ -c_6t &t &c_5\end{pmatrix}
$$
so that $b_5x_1(c_6)n_1^{-1} = x_1(c_5^{-1}c_6)n_1^{-1}b_5'$.

\smallskip\noindent
Step 7:  If $c_5c_7\ne 0$ then, since $v_6x_{-\alpha_0}(c)v_6^{-1}=x_{-\varphi-2\delta}(c)$,
\begin{align*}
u_6v_6b_6x_0(c_7)n_0^{-1}
&=u_6v_6x_0(c_5c_7)n_0^{-1}b_6'
=u_6v_6x_{-\alpha_0}(c_5^{-1}c_7^{-1})x_{\alpha_0}(-c_5c_7)h_{\alpha_0^\vee}(c_5c_7)b_6' \\
&=u_6x_{-\varphi-2\delta}(c_5^{-1}c_7^{-1})v_6b_7
=u_7v_7b_7,
\end{align*}
where
$$u_7=u_6x_{-\varphi-2\delta}(c_5^{-1}c_7^{-1}),\qquad v_7=v_6,\qquad\textrm{and}
$$ 
$$b_6' = \begin{pmatrix}c_5 &-1 &0 \\ 0 &c_5^{-1} &0 \\ 0 &0 &1\end{pmatrix}
\quad\hbox{and}\quad
b_7 =x_{\alpha_0}(-c_5c_7)h_{\alpha_0^{\vee}}(c_5c_7)b_6'= \begin{pmatrix} c_7^{-1} &-c_5^{-1}c_7^{-1} &0 \\ 0 &c_5^{-1} &0 \\
-c_5t &t &c_5c_7\end{pmatrix},
$$
so that $b_6x_0(c_7)n_0^{-1} = x_0(c_5c_7)n_0^{-1}b_6'$.

\smallskip\noindent
Step 8:  No restrictions on $c_5^{-2}c_7^{-1}c_8$.  
Since $v_7x_{\alpha_2}(c)v_7^{-1}=x_{-\alpha_1+\delta}(-c)$,
$$u_7v_7b_7x_2(c_8)n_2^{-1} 
= u_7v_7x_2(c_5^{-2}c_7^{-1}c_8)n_2^{-1}b_7' 
= u_7x_{-\alpha_1+\delta}(-c_5^{-2}c_7^{-1}c_8)v_7n_2^{-1}b_7'
=u_8v_8b_8,
$$
with 
$$
u_8=u_7x_{-\alpha_1+\delta}(-c_5^{-2}c_7^{-1}c_8),\qquad v_8=v_7n_2^{-1} = \begin{pmatrix} 0 &1 &0 \\ 0 &0 &t \\ t^{-1} &0 &0\end{pmatrix},\qquad\textrm{and}
$$
$$
b_8 = b_7' 
= \begin{pmatrix} c_7^{-1} &-c_5^{-1}c_7^{-1}c_8 &c_5^{-1}c_7^{-1} \\
-c_5t &c_5c_7+c_8t &-t \\
-c_5^{-1}c_7^{-1}c_8t &c_5^{-2}c_7^{-1}c_8^2t &c_5^{-1}-c_5^{-2}c_7^{-1}c_8t\end{pmatrix},
$$
so that $b_7x_2(c_8)n_2^{-1}=x_2(c_5^{-2}c_7^{-1}c_8)n_2^{-1}b_7'$.  
%For the record
%$$b_7x_2(c_8)n_2^{-1} = \begin{pmatrix}
%c_7^{-1} &-c_5^{-1}c_7^{-1}c_8 &c_5^{-1}c_7^{-1} \\
%0 &c_5^{-1}c_8 &-c_5^{-1} \\
%-c_5t &c_8t+c_5c_7 &-t\end{pmatrix}$$
%and
%$$x_2(c_5^{-2}c_7^{-1}c_8)n_2^{-1}b_7'
%= \begin{pmatrix} 1 &0 &0 \\ 0 &c_5^{-2}c_7^{-1}c_8 &-1 \\ 0 &1 &0\end{pmatrix}
%\begin{pmatrix} c_7^{-1} &-c_5^{-1}c_7^{-1}c_8 &c_5^{-1}c_7^{-1} \\
%-c_5t &c_8t+c_5c_7 &-t \\
%-c_5^{-1}c_7^{-1}c_8t &c_5^{-2}c_7^{-1}c_8^2t &-c_5^{-2}c_7^{-1}c_8t+c_5^{-1}\end{pmatrix}
%$$

\smallskip\noindent 
Step 9: If $c_5^{-1}c_7c_9-c_5^{-1}c_8=0$ (so $c_9=c_7^{-1}c_8$) then
\begin{align*}
u_8v_8b_8x_0(c_9)n_0^{-1}=u_8v_8x_0(c_5^{-1}c_7c_9-c_5^{-1}c_8)n_0^{-1}b_8'=u_8x_{-\alpha_2-3\delta}(0)v_8n_0^{-1}b_8'=u_9v_9b_9
\end{align*}
with $u_9,v_9$ and $b_9$ as in (\ref{eq:finalstats}).


\begin{thebibliography}{Mac2}


%\bibitem[Ro1]{Ro1} M.A.\ Ronan, \emph{Lectures on buildings}, Perspectives in Mathematics, 
%vol. \textbf{7}, Academic Press Inc., Boston, MA, 1989. MR 1005533 (90j:20001).

%\bibitem[Ro2]{Ro2} M.A.\ Ronan, \emph{A construction of buildings with no rank 3 residue 
%if spherical type}, Buildings and the geometry of diagrams (L.A. Rosati, ed.), 
%Lecture notes in Mathematics, vol. \textbf{1181}, Springer-Verlag, Berlin, 1986. 
%MR 843391 (87h:20077).

%\bibitem[Ba]{Ba} L.M.\ Batten, \emph{Combinatorics of Finite Geometries}, 
%Cambridge University Press, Cambridge, (1986).

\bibitem[BD]{BD} Y.\ Billig and M.\ Dyer, \emph{Decompositions of Bruhat type for the 
Kac-Moody groups},  Nova J.\ Algebra Geom.\ \textbf{3} no.\ 1 (1994), 11--31.

\bibitem[BG]{BG} A.\ Braverman and D.\ Gaitsgory, \emph{Crystals via the affine Grassmannian},  
Duke Math.\ J.\  \textbf{107}  no.\ 3 (2001) 561--575.

\bibitem[CC]{CC} R.\ Carter and Y.\ Chen, \emph{Automorphisms of affine Kac-Moody groups and related Chevalley groups over rings}, J. Algebra, \textbf{155}, 1993, No. 1, 44--94, 
MR1206622(94e:17032).

\bibitem[C]{C} C.\ Curtis, \emph{A further refinement of the Bruhat decomposition},
Proc.\ Amer.\ Math.\ Soc.\  \textbf{102}  no.\ 1(1988) 37--42.

%\bibitem[Ga1]{Ga1} H.\ Garland, \emph{The arithmetic theory of loop groups}, 
%Publ.\ Math.\ IHES \textbf{52} (1980), 181--312.

\bibitem[Ga]{Ga2} H.\ Garland, \emph{A Cartan decomposition for $p$-adic loop groups},
Math.\ Ann.\ \textbf{302} (1995), 151--175.

\bibitem[GK]{GK} D.\ Gaitsgory and D.\ Kazhdan, 
\emph{Representations of algebraic groups over a 2-dimensional local field},  
Geom.\ Funct.\ Anal.\  \textbf{14} no.\ 3 (2004), 535--574.

\bibitem[GL]{GL} S.\ Gaussent and  P.\ Littelmann, \emph{LS galleries, the path model, and 
MV cycles}, Duke Math.\ J.\ \textbf{127} no.\ 1 (2005), 35--88.

\bibitem[GR]{GR} S.\ Gaussent and  G.\ Rousseau, \emph{Kac-Moody groups, hovels 
and Littelmann's paths}, preprint 2007, arXiv:math.GR/0703639.

%\bibitem[IM]{IM}  N.\ Iwahori and H.\  Matsumoto, \emph{On some Bruhat decomposition 
%and the structure of the Hecke rings of $\fp$-adic Chevalley groups},
%Inst.\ Hautes \'Etudes Sci.\ Publ.\ Math.\ \textbf{25} (1965), 5--48. 

\bibitem[Kac]{Kac} V.\ Kac, \emph{Infinite dimensional Lie algebras}, Third edition, Cambridge 
University Press, 1990.

%\bibitem[KM]{KM} M.\ Kapovich and J.J.\ Millson,
%\emph{A path model for geodesics in Euclidean buildings and its applications to representation 
%theory},  arXiv: math.RT/0411182.

%\bibitem[Ku]{Ku} S.\ Kumar, \emph{Kac-Moody groups, their Flag Varieties and 
%Representation Theory}, Progress in Mathematics, vol. \textbf{204}, Birkh\"{a}user, 2002.

%\bibitem[LP]{LP}  C.\ Lenart and A.\ Postnikov,
%\emph{Affine Weyl groups in K-theory and representation theory},
%arXiv: math.RT/0309207. 

%\bibitem[Li1]{Li1} P.\ Littelmann, 
%\emph{A Littlewood-Richardson rule for symmetrizable Kac-Moody algebras}, 
%Invent.\ Math.\ \textbf{116} (1994), 329--346. 

%\bibitem[Li2]{Li2} P.\ Littelmann, 
%\emph{Paths and root operators in representation theory}, 
%Ann.\ Math.\ \textbf{142} (1995), 499--525. 

%\bibitem[Li3]{Li3} P.\ Littelmann, 
%\emph{Characters of representations and paths in $\fh_\RR^*$} , 
%Proc.\ Symp.\ Pure Math.\ \textbf{61} (1997), 29--49. 

\bibitem[Mac1]{Mac1} I.G.\ Macdonald, {\sl Spherical functions on a group of p-adic type}, 
Publ. Ramanujan Institute No. 2, Madras (1971). 

%\bibitem[Mac2]{Mac2} I.G.\ Macdonald, {\sl Symmetric functions and Hall polynomials}, 
%Oxford Mathematical Monographs, Oxford Univ. Press, New York (second edition, 1995)

\bibitem[Mac2]{Mac3} I.G.\ Macdonald, handwritten lecture notes on Kac-Moody algebras, 
1983.

\bibitem[Mac3]{Mac4} I.G.\ Macdonald, \emph{Kac-Moody algebras}, in \textit{Lie Algebras and Related Topics},  Eds D.J. Britten, F.W. Lemire, and R.V. Moody, Conference Proceedings of the 
Canadian Mathematical Society, 1986, Volume 5.

\bibitem[Ra]{Ra} A,\ Ram, \emph{Alcove walks, Hecke algebras, Spherical functions, crystals and column strict tableaux}, Pure Appl.\  Math.\  Quart.\  \textbf{2} no.\ 4 (2006) 963--1013.

\bibitem[Rem]{Rem} B.\ R\'emy, \emph{Groupes de Kac-Moody d\'eploy\'es et presque
d\'eploy\'es}, Ast\'erisque \textbf{277} (2002).

\bibitem[Rou]{Rou} G.\ Rousseau, \emph{Groupes de Kac-Moody d\'eploy\'es sur un corps
local, immeubles microaffines}, Comp.\ Math.\ textbf{142} (2006), 501--528.

%\bibitem[Sc]{Sc} C.\ Schwer, \emph{Galleries, Hall-Littlewood polynomials and structure 
%constants of the spherical Hecke algebra}, arXiv: math.CO/0506287.

\bibitem[St]{St} R.\ Steinberg, \emph{Lecture notes on Chevalley groups}, Yale University, 1967.

%\bibitem[Se]{Se} J.-P.\ Serre, \emph{Arbres, amalgames, $SL_2$}, Ast\'{e}risque, \textbf{46}, 1977.

\bibitem[Ti]{Ti1} J.\ Tits, \emph{Uniqueness and presentation of Kac-Moody groups over fields}, J. Algebra, \textbf{105} No.2, 1987, 542--573, MR0873684(89b:17020).

%\bibitem[Ti2]{Ti2} J.\ Tits, \emph{Buildings of spherical type and finite $BN$-pairs}, Lecture 
%Notes in Mathematics, vol. \textbf{386}, Springer-Verlag, Berlin, 1974. MR 0470099 (57 9866).

%\bibitem[Ti3]{Ti3} J.\ Tits, \emph{Immeubles de type affine}, Buildings and the geometry of diagrams 
%(L.A. Rosati, ed.), Lecture notes in Mathematics, vol. \textbf{1181}, Springer-Verlag, Berlin, 
%1986, pp. 159--190. MR 843391 (87h:20077).

\bibitem[Wak1]{Wak1} M.\  Wakimoto, \emph{Infinite-dimensional Lie algebras},
Translations of Mathematical Monographs \textbf{195} 
American Mathematical Society, Providence, RI, 2001. 
ISBN: 0-8218-2654-9, MR1793723 (2001k:17038)

\bibitem[Wak2]{Wak2} M.\ Wakimoto, \emph{Lectures on infinite-dimensional Lie algebra},
World Scientific Publishing Co., Inc., River Edge, NJ, 2001. ISBN: 981-02-4128-3, 
MR1873994 (2003b:17033) 
 

\end{thebibliography}
\end{document}